\documentclass[
11pt,                          % standard font size
final,                         % draft or final?
english                        % standard language
]{article}

\usepackage[english]{babel}    % with explicit language
\usepackage{amsmath}           % ams mathematical stuff
\usepackage{amssymb}           % ams symbols
\usepackage[utf8]{inputenc}    % smart input of funny chars
\usepackage[T1]{fontenc}       % also for the font encoding
\usepackage[a4paper]{geometry} % geometry of page layout
\usepackage{setspace}          % for onehalfspace
\usepackage{bbm}               % nicer bbm fonts
\usepackage{mathrsfs}          % nice calligraphic font
\usepackage{mathtools}         % some nice tricks for stackrel and DeclarePairedDelimiter
\usepackage{stmaryrd}          % more brackets
\usepackage[amsmath,thmmarks,hyperref]{ntheorem} % nicer theorems
\usepackage{exscale}           % large summation signs in 11pt
\usepackage[sort]{cite}        % nicer citations
\usepackage{xspace}            % better spacing after macros
\usepackage[normalem]{ulem}    % for underlining
\usepackage[activate={true,nocompatibility},
            final,tracking=true,kerning=true,
            spacing=true,factor=1100,
            stretch=10,shrink=10,expansion=false
           ]{microtype}
\microtypecontext{spacing=nonfrench}  % microtype told me to do this...
\usepackage[hyperref]{xcolor}
\definecolor{mydarkblue}{RGB}{0,0,155}
\usepackage[backref=page,      % backrefs in the bibliography
           final=true,         % always treat as final
           colorlinks=true,    % linkcolors are set to a dark blue
           allcolors=mydarkblue,  % this should turn out fine on a black and white printer
           hypertexnames=false,% better counter handling, avoids problems
           plainpages=false,              % whatever that means
           pdfpagelabels=true,            % strange options
           pdfencoding=auto,              % getting worse
           unicode=true                   % unicode is always nice
                      ]{hyperref}         % hyperrefs are cool!
\usepackage[shortcuts]{extdash}  % nonbreaking hyphens

\author{
  \textbf{Matthias Schötz}
  \thanks{Boursier de l'ULB, \href{mailto:Matthias.Schotz@ulb.ac.be}{\texttt{Matthias.Schotz@ulb.ac.be}}.
  This work was supported by the Fonds de la Recherche Scientifique (FNRS) and the Fonds Wetenschappelijk
  Onderzoek - Vlaaderen (FWO) under EOS Project n$^0$30950721.}\\
  Département de Mathématiques\\
  Université libre de Bruxelles
}

\makeatletter

\renewcommand{\mathbb}[1]{\mathbbm{#1}}

\newcommand{\refitem}[1] {\textit{\ref{#1}.)}}

\numberwithin{equation}{section}

\let\originalleft\left
\let\originalright\right
\renewcommand{\left}{\mathopen{}\mathclose\bgroup\originalleft}
\renewcommand{\right}{\aftergroup\egroup\originalright}

\newcommand{\lemmachairxname}{Lemma}
\newcommand{\propositionchairxname}{Proposition}
\newcommand{\theoremchairxname}{Theorem}
\newcommand{\corollarychairxname}{Corollary}
\newcommand{\definitionchairxname}{Definition}
\newcommand{\examplechairxname}{Example}
\newcommand{\proofchairxname}{Proof}

\StartBabelCommands{english}{extras}
\SetString{\lemmachairxname}{Lemma}
\SetString{\propositionchairxname}{Proposition}
\SetString{\theoremchairxname}{Theorem}
\SetString{\corollarychairxname}{Corollary}
\SetString{\definitionchairxname}{Definition}
\SetString{\examplechairxname}{Example}
\SetString{\proofchairxname}{Proof}
\EndBabelCommands

\theoremheaderfont{\normalfont\bfseries}
\theorembodyfont{\itshape}
\newtheorem{lemma}{\lemmachairxname}[section]

\newtheorem{proposition}[lemma]{\propositionchairxname}
\newtheorem{theorem}[lemma]{\theoremchairxname}
\newtheorem{corollary}[lemma]{\corollarychairxname}
\newtheorem{definition}[lemma]{\definitionchairxname}

\theorembodyfont{\rmfamily}

\newtheorem{example}[lemma]{\examplechairxname}

\def\theorem@checkbold{}

\theoremheaderfont{\normalfont\bfseries}
\theorembodyfont{\normalfont}
\theoremstyle{nonumberplain}
\theoremseparator{:}
\theoremsymbol{\hbox{$\boxempty$}}
\newtheorem{proof}{Proof}

\newcommand{\RE}             {\operatorname{\mathsf{Re}}}
\newcommand{\IM}             {\operatorname{\mathsf{Im}}}
\newcommand{\I}              {\mathrm{i}}
\newcommand{\cc}[1]          {\overline{{#1}}}
\newcommand{\Unit}           {\mathbb{1}}
\newcommand{\id}              {\mathrm{id}}
\newcommand{\argument}       {\ignorespaces{\,\cdot\,}\ignorespaces}
\DeclarePairedDelimiter{\abs}{\lvert}{\rvert}
\DeclarePairedDelimiter{\norm}{\lVert}{\rVert}
\newcommand{\hilbert}[1]       {\mathfrak{#1}}
\newcommand{\Stetig}         {\mathscr{C}}
\DeclareFontFamily{U}{FdSymbolF}{}
\DeclareFontShape{U}{FdSymbolF}{m}{n}{
    <-7.1> FdSymbolF-Book
    <7.1-> FdSymbolF-Book
}{}

\DeclareSymbolFont{delimiters}{U}{FdSymbolF}{m}{n}
\DeclareMathDelimiter{\llangle}{\mathopen}{delimiters}{"92}{delimiters}{"92}
\DeclareMathDelimiter{\rrangle}{\mathclose}{delimiters}{"98}{delimiters}{"98}

\DeclarePairedDelimiter{\ordinaryIP}{\langle}{\rangle}

\DeclarePairedDelimiter{\ordinarySet}{\{}{\}}
\DeclarePairedDelimiter{\coolHulls}{\llangle}{\rrangle}

\newcommand{\CC}{\mathbb{C}}
\newcommand{\RR}{\mathbb{R}}
\newcommand{\NN}{\mathbb{N}}

\newcommand{\skal}[3][]{\ordinaryIP[#1]{\,#2 \,#1|\, #3\,}}
\newcommand{\set}[3][]{\ordinarySet[#1]{\,#2 \;#1|\; #3\,}}
\newcommand{\genPos}[2][]{\coolHulls[#1]{\,#2\,}_{\mathrm{pos}}}
\newcommand{\Hermitian}{\textup{H}}

\newcommand{\seminorm}[3][]{\norm[#1]{#3}_{#2}}
\newcommand{\Dom}[1][]{\mathcal{D}_{\smash{#1}}}
\newcommand{\Hilb}{\hilbert{H}}
\newcommand{\algebra}[1]{\mathcal{#1}}
\newcommand{\Lin}{\mathfrak{L}}
\newcommand{\Adbar}{\mathfrak{L}^*}

\newcommand{\A}{\algebra{A}}

\newcommand{\RieszSpace}[1]{\mathcal{#1}}
\newcommand{\R}{\RieszSpace{R}}

\newcommand{\metric}{\mathrm{d}}

\newcommand{\Su}{\textit{Su}}
\newcommand{\Sus}{\textit{Su}$^*$}

\newcommand{\bd}{\mathrm{bd}}

\newcommand{\neu}[1]{\uline{#1}}

\makeatother

\title{%
\texorpdfstring{%
Equivalence of Order and Algebraic Properties\\in Ordered $^*$\=/Algebras%
}{%
Equivalence of Order and Algebraic Properties in Ordered *-Algebras%
}%
}

\date{June 2020}

\begin{document}
\begin{onehalfspace}
\allowdisplaybreaks

\maketitle

\begin{abstract}
  The aim of this article is to describe a class of $^*$\=/algebras
  that allows to treat well-behaved algebras of unbounded operators
  independently of a representation. To this end, Archimedean
  ordered $^*$\=/algebras ($^*$\=/algebras whose 
  real linear subspace of Hermitian elements are an Archimedean ordered vector space
  with rather weak compatibilities with the algebraic structure) are examined. 
  The order induces a translation-invariant uniform metric which comes from a 
  $C^*$\=/norm in the bounded case. It will then be shown that uniformly complete Archimedean ordered
  $^*$\=/algebras have good order properties (like existence of infima, suprema or absolute values)
  if and only if they have good algebraic properties (like existence of inverses
  or square roots). This suggests the definition of \Sus\=/algebras as uniformly complete Archimedean
  ordered $^*$\=/algebras which have all these equivalent properties.
  All methods used are completely elementary and do not require any representation
  theory and not even any assumptions of boundedness, so \Sus\=/algebras generalize 
  some important properties of $C^*$\=/algebras to algebras of unbounded operators. 
  Similarly, they generalize some properties of $\Phi$-algebras (certain lattice-ordered
  commutative real algebras) to non-commutative ordered $^*$\=/algebras.
  It is also shown that the order on \Sus\=/algebra is uniquely determined, so
  \Sus\=/algebras are indeed just a class of well-behaved $^*$\=/algebras.
  As an example, \Sus\=/algebras of unbounded operators on a Hilbert space are 
  constructed. They arise e.g. as $^*$\=/algebras of symmetries of a self-adjoint
  (not necessarily bounded) Hamiltonian operator of a quantum mechanical system.
\end{abstract}

%  The aim of this article is to describe a class of *-algebras that allows to treat well-behaved algebras of unbounded operators independently of a representation. To this end, Archimedean ordered *-algebras (*-algebras whose real linear subspace of Hermitian elements are an Archimedean ordered vector space with rather weak compatibilities with the algebraic structure) are examined. The order induces a translation-invariant uniform metric which comes from a C*-norm in the bounded case. It will then be shown that uniformly complete Archimedean ordered *-algebras have good order properties (like existence of infima, suprema or absolute values) if and only if they have good algebraic properties (like existence of inverses or square roots). This suggests the definition of Su*-algebras as uniformly complete Archimedean ordered *-algebras which have all these equivalent properties. All methods used are completely elementary and do not require any representation theory and not even any assumptions of boundedness, so Su*-algebras generalize some important properties of C^*-algebras to algebras of unbounded operators. Similarly, they generalize some properties of Phi-algebras (certain lattice-ordered commutative real algebras) to non-commutative ordered *-algebras. As an example, Su*-algebras of unbounded operators on a Hilbert space are constructed. They arise e.g. as *-algebras of symmetries of a self-adjoint (not necessarily bounded) Hamiltonian operator of a quantum mechanical system.
 
\section{Introduction}
\label{sec:Introduction}
Many important examples of $^*$\=/algebras, especially $^*$\=/algebras of complex-valued
functions or $^*$\=/algebras of adjointable endomorphisms, carry a partial order on their Hermitian
elements that is compatible with the algebraic structure: 
In the former case, this is the order by pointwise comparison of real-valued
functions, in the latter it is the usual order on Hermitian operators.
From a more abstract point of view, it has long been known that there exists an
intrinsic partial order on the Hermitian elements of a $C^*$\=/algebra, which can be
defined in many equivalent ways (e.g. by declaring squares of Hermitian elements
to be the positive ones, or elements with non-negative real spectrum). This is
of course not surprising as $C^*$\=/algebras can always be represented as $^*$\=/algebras
of bounded operators, and in the commutative case even as $^*$\=/algebras
of continuous functions. However, a generalization
of this approach seems to be difficult, at least in the realm of topological $^*$\=/algebras:
Already Banach $^*$\=/algebras can have extremely pathological order properties.
Because of this, examining
$^*$\=/algebras having in some sense ``unbounded'' elements by means of locally convex
$^*$\=/algebras is a rather hard task.

However, in the commutative case, $^*$\=/algebras are just the complexifications of
real associative algebras. So the theory of ordered real algebras, especially
of lattice ordered ones like (almost) $f$\=/algebras and $\Phi$\=/algebras, immediately
carries over and yields examples of well-behaved ordered $^*$\=/algebras even beyond
the scope of $C^*$\=/algebras. The representation theorem \cite[Thm.~2.3]{henriksen.johnson:structureOfArchimedeanLatticeOrderedAlgebras}
for $\Phi$\=/algebras as algebras of functions on a compact Hausdorff space with values in the extended
real numbers further exemplifies the close relation between commutative $C^*$\=/algebras
and (complexifications of) $\Phi$\=/algebras.

The aim of the present article is to examine ordered $^*$\=/algebras
and ultimately to determine a class of very well-behaved ordered $^*$\=/algebras
that generalize important properties of $C^*$\=/algebras to the unbounded case,
as well as properties of $\Phi$\=/algebras to the non-commutative case. This includes
the existence of suprema and infima of finitely many commuting Hermitian elements,
of absolute values, square roots of positive elements and inverses of elements that
are coercive (i.e. ``strictly'' positive), as well as automatic continuity
of unital $^*$\=/homomorphisms and the uniqueness of the order. Special attention 
is given to situations where order-theoretic and algebraic concepts are equivalent.
The most obvious example for this are absolute values: The absolute value $\abs{a}$ of a Hermitian
element $a$ should be, from the purely order-theoretic point of view, the supremum of $a$ and $-a$.
But from a more algebraic point of view, $\abs{a}$ should be the (positive) square root of $a^2$.
This raises the question whether, or under which circumstances, the two descriptions
are equivalent.

It will be shown that every Archimedean ordered $^*$\=/algebra carries a metrizable,
translation-invariant topology. In the bounded case, this topology comes from a
$C^*$-norm as in \cite{scheiderer:positivityAndSOSaGuide}, but there is no need for restriction to this special case: Theorem~\ref{theorem:su}
shows that for complete Archimedean ordered $^*$\=/algebras, the first properties 
mentioned above (from existence of suprema and infima to existence of inverses) 
are all equivalent and then imply the others (automatic continuity, uniqueness of the 
order and compatibility or equivalence of some further order-theoretic 
and algebraic concepts). Those algebras where these equivalent properties are
fulfilled will be called \Sus\=/algebras. They include $C^*$\=/algebras as well 
as (complexifications of) complete $\Phi$\=/algebras as special cases. In the end, examples
of \Sus\=/algebras of unbounded operators, which are neither $C^*$- nor $\Phi$\=/algebras,
will be constructed. This way, \Sus\=/algebras allow to examine $^*$\=/algebras
of unbounded operators independently of concrete representations.

The article is organized as follows: The next Section~\ref{sec:preliminaries}
explains the notation and gives some basic and well-known
facts especially about ordered vector spaces. Section~\ref{sec:definition}
contains the definition of (quasi\=/) ordered $^*$\=/algebras as well as some important examples,
both well-behaved and ill-behaved ones, and discusses the construction of the
uniform metric. After that, Section~\ref{sec:rad} describes radical ordered $^*$\=/algebras
which fulfil an additional compatibility between multiplication and order, and it is
shown that especially all symmetric ordered $^*$\=/algebras (those in which ``strictly''
positive elements have a multiplicative inverse) are radical.
The operations $\vee$ and $\wedge$, which describe
especially well-behaved suprema and infima of two commuting Hermitian elements,
are discussed in Section~\ref{sec:supsetc}. This leads to the definition of $\Phi^*$\=/algebras
which are essentially non-commutative generalizations of $\Phi$\=/algebras.
Square roots, which allow the construction of absolute values and thus of suprema and infima,
are examined in Section~\ref{sec:sqrt}. All this then leads to
the main Theorem~\ref{theorem:su} in Section~\ref{sec:su}, which essentially
states that in the uniformly complete case, the existence of suprema, infima, absolute values,
square roots and inverses are equivalent, and motivates Definition~\ref{definition:sus} of \Sus\=/algebras
as those complete ordered $^*$\=/algebras where these equivalent conditions are
fulfilled. Moreover, all the results obtained in the previous sections (like uniqueness
of the order or automatic continuity of unital $^*$\=/homomorphisms) then apply
especially to these \Sus\=/algebras.
Finally, in Section~\ref{sec:example}, examples of \Sus\=/algebras
of unbounded operators on a Hilbert space are constructed.
\section{Preliminaries} \label{sec:preliminaries}
The natural numbers are $\NN = \{1,2,3,\dots\}$, $\NN_0 \coloneqq \NN\cup\{0\}$ and 
the sets of real and complex numbers are denoted by $\RR$ and $\CC$, respectively.
If $X$ is a set, then $\id_X\colon X\to X$ is $x\mapsto \id_X(x):= x$. 
A \neu{quasi-order} on $X$ is a reflexive
and transitive relation, hence a \neu{partial order} is a quasi-order that is additionally
anti-symmetric. If $X$ and $Y$ are both endowed with a quasi-order $\lesssim$, then a map
$\Psi \colon X \to Y$ is called \neu{increasing} if $\Psi(x) \lesssim \Psi(\tilde{x})$ for all $x,\tilde{x}\in X$
with $x\lesssim \tilde{x}$. If $\Psi$ is injective and increasing and if conversely also $x\lesssim \tilde{x}$
holds for all $x,\tilde{x}\in X$ with $\Psi(x) \lesssim \Psi(\tilde{x})$, then $\Psi$ is called
an \neu{order embedding}.

A \neu{quasi-ordered vector space} is a real vector space $V$ endowed with a quasi-order
$\lesssim$ such that $u + w \lesssim v + w$ and $\lambda u \lesssim \lambda v$ hold for
all $u,v,w\in V$ with $u\lesssim v$ and all $\lambda \in {[0,\infty[}$. An 
\neu{ordered vector space} is a quasi-ordered vector space whose order is even a partial order,
which is then typically denoted by $\le$ instead of $\lesssim$.
For every quasi-ordered vector space $V$, the convex cone (non-empty subset of a real vector space
closed under addition and scalar multiplication with non-negative reals) of \neu{positive}
elements is $V^+ \coloneqq \set{v\in V}{v\gtrsim 0}$, and one can check that this
describes a one-to-one correspondence between convex cones in $V$ and orders on $V$ that turn $V$
into a quasi-ordered vector space. From this point of view, $V$ is even an ordered vector
space if and only if $V^+\cap (-V^+) = \{0\}$. 
A quasi-ordered vector space $V$ is called \neu{Archimedean} if it has the following
property: Whenever $v \lesssim \epsilon w$ holds for fixed $v \in V$ and $w\in V$
and all $\epsilon\in{]0,\infty[}$, then $v \lesssim 0$.

The real vector space $\Lin(V,W)$ of all 
linear maps $\Psi \colon V\to W$ between two quasi-ordered vector spaces is again a
quasi-ordered vector space by declaring the positive elements to be precisely the 
increasing linear maps. Because of this, the increasing linear maps are called \neu{positive}.
Note that a linear map $\Psi\colon V\to W$ is increasing if and only if $\Psi(v) \in W^+$
for all $v\in V^+$.

In ordered vector spaces it makes sense to discuss suprema and infima of arbitrary 
non-empty subsets.
A \neu{Riesz space} (or vector lattice) is an ordered vector space $\RieszSpace{R}$ in which 
suprema and infima of all pairs of elements exist. It is well-known that this is already
the case if $\sup \{r,-r\}$ exists for all $r \in \RieszSpace{R}$. The usual notations
$\vee$ and $\wedge$ for suprema and infima in Riesz spaces will be avoided and will
be reserved for a similar concept that is introduced later for ordered $^*$-algebras.
Endowing Riesz spaces with an additional algebraic structure leads to e.g.\ the concept of
\neu{$\Phi$-algebras}, which are Archimedean Riesz spaces $\R$ endowed with a multiplication 
that turns $\R$ into a real unital associative algebra such that $rs \in \R^+$ for all $r,s\in \R^+$
and $\inf\{rt, s\} = \inf\{tr, s\} = 0$ for all $r,s,t \in \R^+$ with $\inf \{ r, s \} = 0$.
Note that this property, applied twice with $t=r$ and $t=s$, especially implies that
$rs = \inf\{rs,rs\} = 0$ for all $r,s \in \R^+$ with $\inf \{ r, s \} = 0$, and from
$0 = \inf\big\{\sup\{r,-r\}-r,\sup\{r,-r\}+r\big\}$ for all $r\in \mathcal{R}$ it follows that $(\sup\{r,-r\})^2 = r^2$.
One remarkable result about $\Phi$\=/algebras is a representation theorem as algebras of extended
real-valued functions on compact Hausdorff spaces
\cite{henriksen.johnson:structureOfArchimedeanLatticeOrderedAlgebras}, and especially that
$\Phi$\=/algebras are automatically commutative. This shows that $\Phi$\=/algebras
are a good abstraction of lattice-ordered algebras of real-valued functions. There are also
many similar notions of Riesz spaces with multiplication that have been studied extensively,
most notably (almost) $f$\=/algebras. One essential takeaway is that a multiplication
on an Archimedean Riesz space is automatically commutative under very mild assumptions of 
compatibility with the order \cite{kouki.toumi.toumi:commutativityOfSomeArchimedeanOrderedAlgebras, 
scheffold:ffBanachverbandsalgebren, bernau.huijsmans:almostFAlgebrasAndDAlgebras,
buskes.vanRooij:AlmostFAlgebrasCommutativityandCSInequality}.
This unfortunately means that such algebras are not suitable for the description of
reasonably well-behaved non-commutative algebras of operators, which are the usual
non-commutative analog of algebras of functions.

A \neu{$^*$\=/vector space} is a complex vector space $V$ endowed with an antilinear involution
$\argument^*\colon V\to V$. An element $v$ of a $^*$\=/vector space $V$ is called \neu{Hermitian}
if $v=v^*$ and the real linear subspace of Hermitian elements in $V$ is denoted by $V_\Hermitian$.
Then $V = V_\Hermitian \oplus \I V_\Hermitian$ as a real vector space, and this decomposition
can explicitly be described as $v = \RE(v)+\I\,\IM(v)$ with $\RE(v) = \frac{1}{2}(v+v^*)$ and 
$\IM(v) = \frac{1}{2\I}(v-v^*)$ for all $v\in V$. The most obvious example of a $^*$\=/vector space
is of course given by $\CC$ with complex conjugation $\cc{\argument}$ as $^*$\=/involution.
The complex vector space $\Lin(V,W)$ of all linear maps $\Psi \colon V\to W$ between two
$^*$\=/vector spaces is again endowed with an antilinear involution defined by 
$\Psi^*(v) \coloneqq \Psi(v^*)^*$ for all $\Psi \in \Lin(V,W)$ and all $v\in V$.
A linear map $\Psi \colon V\to W$ thus is Hermitian if and only if $\Psi(v^*) = \Psi(v)^*$
holds for all $v\in V$, or equivalently, if and only if $\Psi(v) \in W_\Hermitian$ for all $v\in V_\Hermitian$.

A \neu{$^*$\=/algebra} is a unital associative complex algebra $\A$ which is also a
$^*$\=/vector space such that $(ab)^* = b^*a^*$ holds for all $a,b\in \A$. Its unit is denoted
by $\Unit$ or, more explicitly, by $\Unit_\A$, and is automatically Hermitian. Moreover, a
\neu{unital $^*$\=/homomorphism} between two $^*$\=/algebras is a unital homomorphism of algebras
which is additionally Hermitian, and a \neu{unital $^*$\=/subalgebra} of a $^*$\=/algebra is a unital subalgebra
that is stable under $\argument^*$. It is not explicitly required that $0\neq \Unit$,
but the only case in which this is not fulfilled is the not very interesting algebra $\{0\}$.
For a subset $S\subseteq \A$ of a $^*$\=/algebra $\A$, the \neu{commutant} $S' \coloneqq \set[\big]{a\in \A}{\forall_{s\in S}: sa=as}$
is a unital subalgebra, and even a unital $^*$\=/subalgebra if $S$ is stable under $\argument^*$.
If $S$ is commutative, then the \neu{bicommutant} $S''$ is again commutative and $S \subseteq S'' \subseteq S'$.
For example, the multiplicative inverse $a^{-1}$ of an invertible $a\in \A$ is in the
bicommutant of $a$. A \neu{$C^*$-(semi)norm} on a $^*$\=/algebra
$\A$ is a (semi)norm $\seminorm{}{\argument}$ for which $\seminorm{}{ab} \le \seminorm{}{a}\seminorm{}{b}$
and $\seminorm{}{a^*a} = \seminorm{}{a}^2$ hold for all $a,b\in \algebra{A}$,
hence especially $\seminorm{}{a^*}=\seminorm{}{a}$,
and a \neu{$C^*$-algebra} 
is a $^*$\=/algebra that is complete with respect to the topology of a $C^*$-norm.

A \neu{(quasi-)ordered $^*$\=/vector space} is a $^*$\=/vector space $V$ whose
real linear subspace of Hermitian elements $V_\Hermitian$ is endowed with an order
that turns it into a (quasi-)ordered vector space. The properties of ordered vector
spaces and linear functions between them, like being Archimedean or positive, apply 
to ordered $^*$\=/vector spaces in the obvious way, i.e. they refer
to the order on the Hermitian elements.
\section{Archimedean Ordered \texorpdfstring{$^*$\=/Algebras}{*-Algebras}} \label{sec:definition}
(Quasi-)ordered $^*$\=/algebras are defined analogously to (quasi-)ordered $^*$\=/vector spaces,
and have already been studied in e.g.~\cite{powers:SelfadjointAlgebrasOfUnboundedOperators2, schmuedgen:UnboundedOperatorAlgebraAndRepresentationTheory} as ``$^*$\=/algebras that
are equipped with an admissible wedge'' in the context of $^*$-algebras of unbounded
operators and, especially in the commutative case and up to complexification, as 
``rings equipped with a quadratic module'' in real algebraic geometry, see e.g.\ \cite{scheiderer:positivityAndSOSaGuide}
for a survey. However, it is important to point out that with respect to
quadratic modules, the term ``Archimedean'' unfortunately is used in a different way than
with respect to ordered vector spaces.
\begin{definition}
  A \neu{quasi-ordered $^*$\=/algebra} is a $^*$\=/algebra $\A$ whose real linear 
  subspace $\A_\Hermitian$ is endowed with a quasi-order $\lesssim$ such that
  \begin{equation}
    a + c \lesssim b +c
    \,,\quad\quad
    d^*a\,d \lesssim d^*b\,d
    \quad\quad\text{and}\quad\quad
    0 \lesssim \Unit
  \end{equation}
  hold for all $a,b,c\in \A_\Hermitian$ with $a\lesssim b$ and all $d\in \A$. An 
  \neu{ordered $^*$\=/algebra} is a quasi-ordered $^*$\=/algebra $\A$ for which $\A_\Hermitian$
  is partially ordered.
\end{definition}
As $^*$\=/algebras are required to have a unit, these axioms especially imply that every 
(quasi\=/)ordered $^*$\=/algebra is a (quasi\=/)ordered $^*$\=/vector space. Thus, a quasi-ordered 
$^*$\=/algebra $\A$ will be called Archimedean if $\A_\Hermitian$ is Archimedean as a quasi-ordered vector space
and we will especially be interested in positive Hermitian linear maps and positive unital $^*$\=/homomorphisms
between quasi-ordered $^*$-algebras.
Note that the set $\A_\Hermitian^+$ of positive Hermitian elements of $\A$ generates $\A_\Hermitian$
as a real vector space because $4a = (a+\Unit)^2 - (a-\Unit)^2$ holds for all $a\in \A_\Hermitian$ and because $(a \pm \Unit)^2 \in \A_\Hermitian^+$.
Moreover, one easily checks that
\begin{equation}
  \lambda a + \mu b \in \A_\Hermitian^+
  \,,\quad\quad
  d^*a\,d \in \A_\Hermitian^+
  \quad\quad\text{and}\quad\quad
  \Unit \in \A_\Hermitian^+
  \label{eq:admissibleCone}
\end{equation}
hold for all $a,b \in \A_\Hermitian^+$, $d\in \A$ and scalars $\lambda, \mu \in {[0,\infty[}$.
Conversely, if $\A$ is a $^*$\=/algebra and $\A_\Hermitian^+$ an arbitrary subset of $\A_\Hermitian$
that fulfils these three conditions \eqref{eq:admissibleCone}, then there is a unique order on $\A_\Hermitian$
such that $\A$ becomes a quasi-ordered $^*$-algebra whose set of positive Hermitian elements 
is precisely this set $\A_\Hermitian^+$. This order is given for $a,b\in \A_\Hermitian$ by $a\lesssim b$ iff $b-a\in\A_\Hermitian^+$.

Again, the most obvious example of an ordered $^*$\=/algebra is $\CC$ with the usual order on $\CC_\Hermitian \cong \RR$.
More interesting ones are:
\begin{example} \label{example:fun}
  Let $X$ be a non-empty set and $\CC^X$ the unital $^*$\=/algebra of all complex-valued functions
  on $X$ with the pointwise operations. Then $\CC^X$ with the pointwise order on its Hermitian elements,
  i.e. $f \le g$ if and only if $f(x) \le g(x)$ for all $x\in X$, is an Archimedean ordered $^*$\=/algebra.
  Consequently, all unital $^*$\=/subalgebras of $\CC^X$ with this pointwise order are
  Archimedean ordered $^*$\=/algebras as well.
\end{example}
Special cases of such ordered $^*$\=/algebras of functions are of course
those of continuous functions, denoted by $\Stetig(X)$ if $X$ is a topological space.
Another special case are polynomials, which demonstrate that there can be, in general,
many possible orders on the same $^*$\=/algebra:
\begin{example} \label{example:poly}
  Let $\CC[x_1,\dots,x_N]$ with $N\in \NN$ be the $^*$\=/algebra of complex polynomials in 
  $N$ Hermitian variables $x_1,\dots,x_N$, i.e.\ the $^*$\=/involution is given by complex conjugation of all coefficients.
  For every subset $S \subseteq \RR^N$, the \neu{$S$-pointwise order}
  on the Hermitian polynomials, i.e.
  $p \le q$ if and only if $p(s_1,\dots,s_N) \le q(s_1,\dots,s_N)$ for all $(s_1,\dots,s_N)\in S$,
  turns $\CC[x_1,\dots,x_N]$ into an Archimedean quasi-ordered $^*$\=/algebra, which is even
  an Archimedean ordered $^*$\=/algebra e.g.\ if $S$ has non-empty interior.
\end{example}
Non-commutative examples are provided by $^*$\=/algebras of operators, i.e. $O^*$\=/algebras:
\begin{example} \label{example:Ostar}
  Let $\Dom$ be a complex pre-Hilbert space with inner product $\skal{\argument}{\argument}$
  (antilinear in the first, linear in the second argument), then a linear endomorphism $a \colon \Dom \to \Dom$
  is said to be \neu{adjointable} if there exists a (necessarily unique) linear
  map $a^*\colon \Dom\to\Dom$ such that $\skal{a^*(\xi)}{\eta} = \skal{\xi}{a(\eta)}$
  holds for all $\xi,\eta\in \Dom$. In this case, $a^*$ is called the \neu{adjoint endomorphism}.
  The set of all adjointable linear endomorphisms on $\Dom$ is a $^*$\=/algebra and becomes
  an Archimedean ordered $^*$\=/algebra, denoted $\Adbar(\Dom)$, with the usual order of Hermitian operators on $\Dom$,
  i.e. $a \le b$ if and only if $\skal{\xi}{a(\xi)} \le \skal{\xi}{b(\xi)}$ for all $\xi \in \Dom$.
  Consequently, all unital $^*$\=/subalgebras of $\Adbar(\Dom)$ are Archimedean ordered $^*$\=/algebras as well.
  These are the $O^*$\=/algebras on $\Dom$, see e.g. the monograph 
  \cite{schmuedgen:UnboundedOperatorAlgebraAndRepresentationTheory} for more details.
\end{example}
In these first examples, the order is essentially determined by positive Hermitian linear functionals,
namely the evaluation functionals at points of $X$ and $S$ in Examples~\ref{example:fun} and
\ref{example:poly}, or the vector functionals $\Adbar(\Dom) \ni a \mapsto \skal{\xi}{a(\xi)} \in \CC$
in Example \ref{example:Ostar}. But there are also other types of examples:
\begin{example} \label{example:gen}
  Let $\algebra{A}$ be a unital $^*$\=/algebra, $G\subseteq \algebra{A}_\Hermitian$ and define
  $\genPos{G}$ as
  \begin{equation}
    \genPos{G} \coloneqq
    \set[\Big]{
      \sum\nolimits_{n=1}^{N} a_{n}^*\,g_n\,a_{n}
    }{
      N \in \NN;\,g_1,\dots,g_N\in G\cup\{\Unit\},\, a_{1},\dots,a_N \in \algebra{A}
    }\,.
  \end{equation}
  Then setting $\A^+_\Hermitian \coloneqq \genPos{G}$
  turns $\A$ into a quasi-ordered $^*$\=/algebra. This order on $\A$ will be called the
  \neu{order generated by $G$}, and $\genPos{G}$ is the smallest
  (with respect to inclusion) choice of positive Hermitian elements that contains $G$ and 
  with which $\A$ becomes a quasi-ordered $^*$\=/algebra. Especially for
  $G=\emptyset$ we write
  \begin{equation}
    \A^{++}_\Hermitian \coloneqq \genPos{\emptyset} = \set[\Big]{\sum\nolimits_{n=1}^N a_n^*a_n}{N\in \NN;\,a_1,\dots,a_N \in \A}
  \end{equation}
  and call the elements of $\A^{++}_\Hermitian$ \neu{algebraically positive}.
\end{example}
There are many strong results (called ``Positivstellensätze'') in real algebraic geometry that
link orders on $^*$\=/algebras that are induced by positive Hermitian linear functionals like
in Examples~\ref{example:fun}, \ref{example:poly} and \ref{example:Ostar} to orders generated
by a set of positive elements like in Example~\ref{example:gen}. These include the classical
Positivstellensatz of Krivine and Stengle for polynomials but also newer results for non-commutative
$^*$\=/algebras like \cite{schmuedgen:StrictPositivstellensatzForWeylAlgebra, schmuedgen:StrictPositivstellensatzForEnvelopingAlgebras}.

Choosing the order on a $^*$\=/algebra $\A$ for which $\A^+_\Hermitian = \A^{++}_\Hermitian$
yields a canonical way to construct a quasi-ordered $^*$\=/algebra out of any $^*$\=/algebra.
For example, the canonical order on $C^*$-algebras can be described like this. There is also
another canonical (yet pathological) choice, namely $\A^+_\Hermitian = \A_\Hermitian$.
For a general quasi-ordered $^*$\=/algebra $\A$ one clearly has 
$\A^{++}_\Hermitian \subseteq \A^+_\Hermitian \subseteq \A_\Hermitian$, but
it is well-known that these extreme cases may coincide:
\begin{example} \label{example:negativeUnit}
  Let $\mathbb{S} \coloneqq \set{z\in \CC}{\abs{z} = 1}$ and let $\A$ be the unital associative
  algebra $\A \coloneqq \Stetig(\mathbb{S})$, but endowed with the $^*$\=/involution $f^* \coloneqq
  \cc{\argument} \circ f \circ \tau$ for all $f\in \Stetig(\mathbb{S})$ (instead of the usual pointwise one $f^* \coloneqq \cc{\argument} \circ f$),
  where $\tau \colon \mathbb{S}\to \mathbb{S}$
  is $z \mapsto \tau(z) \coloneqq -z$. This way $\Stetig(\mathbb{S})$
  indeed becomes a $^*$\=/algebra. The usual norm $\seminorm{\max}{f} \coloneqq \max_{z\in \mathbb{S}} \abs{f(z)}$
  turns $\Stetig(\mathbb{S})$ into a Banach space and makes multiplication and $^*$\=/involution
  continuous. However, $\id_{\mathbb{S}}$ describes a function in $\Stetig(\mathbb{S})$
  for which $- \Unit_{\mathbb{S}} = (\id_{\mathbb{S}})^* \,\id_{\mathbb{S}} \in \A^{++}_\Hermitian$
  holds, thus $\A_\Hermitian = \A_\Hermitian^{++} - \A_\Hermitian^{++} = \A_\Hermitian^{++}$.
\end{example}
Finally, there is a standard example of a non-Archimedean ordered $^*$\=/algebra:
\begin{example} \label{example:ptlg}
  The commutative unital subalgebra
  \begin{equation}
    \A
    \coloneqq
    \set[\bigg]{ M_{a,b} \coloneqq \binom{\,a\,\,\,b\,}{\,0\,\,\,a\,} \in \CC^{2\times 2} }{a,b\in \CC}
  \end{equation}
  of the matrix algebra $\CC^{2\times 2}$ with elementwise complex conjugation as 
  $^*$\=/involution becomes a $^*$\=/algebra. Its algebraically positive elements are
  \begin{equation}
    \A^{++}_\Hermitian
    =
    \set[\big]{ M_{a,b} }{a,b\in \RR\text{ with } a>0\text{ or } a=b=0}
  \end{equation}
  and $\A$ with the algebraic order is an ordered $^*$\=/algebra, but not Archimedean because
  $M_{0,1} \le \epsilon M_{1,0}$ for all $\epsilon \in {]0,\infty[}$. Note especially
  that $M_{0,1}$ is a non-zero Hermitian element that squares to $0$.
\end{example}
Examples~\ref{example:negativeUnit} and \ref{example:ptlg} already
indicate that one should not expect to be able to prove many strong results for
quasi-ordered $^*$-algebras without any additional assumptions. Nevertheless,
there is at least a possibility to characterize the pathological elements and
in many cases one can eventually get rid of them by taking a suitable quotient.
This way it will become clear that an order on a $^*$-algebra can be seen as a
generalization of a $C^*$-norm. This follows essentially \cite{cimpric:representationTheoremForArchimedeanQuadraticModules},
but caution is advised because of the different usage of the term ``Archimedean'' there:
\begin{lemma} \label{lemma:ordersquare}
  Let $\A$ be a quasi-ordered $^*$\=/algebra, $a\in \A_\Hermitian$ and $\lambda \in {]0,\infty[}$,
  then $a^2 \lesssim \lambda^2 \Unit$ if and only if $-\lambda \Unit \lesssim a \lesssim \lambda \Unit$.
  If $\A$ is Archimedean, then this equivalence also holds for $\lambda = 0$.
\end{lemma}
\begin{proof}
  For $\lambda \in {]0,\infty[}$ this is essentially \cite[Lemma~3.1]{cimpric:representationTheoremForArchimedeanQuadraticModules},
  an immediate consequence of the identities
  \begin{align*}
    \lambda \Unit \pm a
    &=
    \frac{\lambda^2 \Unit - a^2 + (\lambda \Unit \pm a)^2}{2\lambda}
    \shortintertext{and}
    \lambda^2 \Unit - a^2
    &=
    \frac{(\lambda \Unit +a)(\lambda \Unit -a)(\lambda \Unit +a)+(\lambda \Unit -a)(\lambda \Unit +a)(\lambda \Unit -a)}{2\lambda}
    .
  \end{align*}
  So $a^2 \lesssim 0$ implies $-\epsilon \Unit \lesssim a \lesssim \epsilon \Unit$
  and $0 \lesssim a \lesssim 0$ implies $a^2 \lesssim \epsilon^2 \Unit$
  for all $\epsilon \in {]0,\infty[}$. If $\A$ is Archimedean, then this shows that 
  $a^2 \lesssim 0$ and $0 \lesssim a \lesssim 0$ are also equivalent.
\end{proof}
\begin{proposition} \label{proposition:nonilpotent}
  Let $\A$ be an Archimedean ordered $^*$\=/algebra and $a\in \A_\Hermitian$ nilpotent, then $a=0$.
\end{proposition}
\begin{proof}
  Let $n \in \NN$ be the minimal exponent for which $a^n = 0$. Then $n$ is odd because otherwise
  $0 \le a^{n/2} \le 0$ by the previous Lemma~\ref{lemma:ordersquare}, which contradicts
  minimality of $n$. But $a^{n+1} = 0$ now implies $0 \le a^{(n+1)/2} \le 0$, so $n=1$ by minimality
  of $n$.
\end{proof}
In Example~\ref{example:ptlg} we have seen that nilpotent Hermitian elements can indeed occur in non-Archimedean ordered $^*$\=/algebras.
Like in \cite{cimpric:representationTheoremForArchimedeanQuadraticModules} we define:
\begin{definition} \label{definition:inftyseminorm}
  Let $\A$ be a quasi-ordered $^*$\=/algebra, then define the map $\seminorm{\infty}{\argument} \colon \A \to [0,\infty]$,
  \begin{equation}
    a \mapsto \seminorm{\infty}{a} \coloneqq \inf \set[\big]{\lambda \in {]0,\infty[}}{a^*a \lesssim \lambda^2 \Unit}\,,
  \end{equation}
  where it is understood that the infimum of the empty set is $\infty$. An element
  $a\in \A$ is called \neu{uniformly bounded} if $\seminorm{\infty}{a} < \infty$ 
  and the set of all uniformly bounded elements in $\A$ is denoted by $\A^\bd$.
  The algebra $\A$ itself is called \neu{uniformly bounded} if $\A = \A^\bd$.
\end{definition}
Lemma~\ref{lemma:ordersquare} immediately gives an alternative description of $\seminorm{\infty}{\argument}$ 
on Hermitian elements:
\begin{proposition} \label{proposition:inftyseminorm}
  Let $\A$ be a quasi-ordered $^*$\=/algebra and $a\in \A_\Hermitian$, then
  \begin{equation}
    \seminorm{\infty}{a} = \inf\set[\big]{\lambda \in {]0,\infty[}}{-\lambda \Unit \lesssim a\lesssim \lambda \Unit}\,,
  \end{equation}
  where again the infimum of the empty set is $\infty$.
\end{proposition}
In the Archimedean case, these infima are even minima:
\begin{proposition} \label{proposition:inftyseminormIsmin}
  Let $\A$ be an Archimedean quasi-ordered $^*$\=/algebra and $a\in \A^\bd$, then
  \begin{equation}
    a^*a \lesssim \seminorm{\infty}{a}^2 \Unit\,.
  \end{equation}
  If even $a\in (\A^\bd)_\Hermitian$, then also
  \begin{equation}
    -\seminorm{\infty}{a} \Unit \lesssim a \lesssim \seminorm{\infty}{a} \Unit\,.
  \end{equation}
\end{proposition}
\begin{proof}
  From the definition of $\seminorm{\infty}{a}$ one sees that
  $a^*a \lesssim \big(\seminorm{\infty}{a}^2 + \epsilon\big) \Unit$
  for all $\epsilon \in {]0,\infty[}$, hence
  $a^*a \lesssim \seminorm{\infty}{a}^2 \Unit$ as $\A$ is Archimedean.
  If $a$ is even Hermitian, then this implies 
  $-\seminorm{\infty}{a} \Unit \lesssim a \lesssim \seminorm{\infty}{a} \Unit$
  by Lemma~\ref{lemma:ordersquare} again.
\end{proof}
The crucial property of $\seminorm{\infty}{\argument}$ is 
that it yields a $C^*$-(semi)norm on the uniformly bounded elements. Recall that
a \neu{$^*$\=/ideal} of a $^*$\=/algebra $\A$ is a linear subspace $\algebra{I}\subseteq \A$
that is stable under the $^*$\=/involution and fulfils $ba \in \algebra{I}$ for all $a\in \A$
and all $b\in\algebra{I}$ (thus also $ab = (b^* a^*)^* \in \algebra{I}$ for all $a\in \A$
and all $b\in\algebra{I}$).
\begin{proposition} \label{proposition:Cstar}
  Let $\A$ be a quasi-ordered $^*$\=/algebra, then $\A^\bd$ is a unital $^*$\=/subalgebra 
  of $\A$, and the restriction of $\seminorm{\infty}{\argument}$ to $\A^\bd$ is 
  a $C^*$-seminorm. Its kernel $\mathcal{K} \coloneqq \set[\big]{a\in \A}{\seminorm{\infty}{a}=0}$
  is a $^*$\=/ideal of $\A^\bd$. If $\A$ is Archimedean, then $\mathcal{K}$ is even
  a $^*$\=/ideal of whole $\A$, and if $\A$ is an Archimedean ordered $^*$\=/algebra, then
  $\mathcal{K} = \{0\}$ so that $\seminorm{\infty}{\argument}$ is a $C^*$\=/norm on $\A^\bd$.
\end{proposition}
\begin{proof}
  The claims for general, not necessarily Archimedean $\A$ have been proven in \cite[Thm.~3.2]{cimpric:representationTheoremForArchimedeanQuadraticModules}.
  For convenience of the reader, the details are given here as well:
  
  From the definition of $\seminorm{\infty}{\argument}$
  it is clear that $\Unit \in \A^\bd$ with $\seminorm{\infty}{\Unit}\le 1$
  and that $\alpha a \in\A^\bd$ with $\seminorm{\infty}{\alpha a} = \abs{\alpha} \seminorm{\infty}{a}$
  for all $a\in \A^\bd$ and all $\alpha \in \CC\backslash\{0\}$, as well as for $\alpha = 0$ because clearly
  $\seminorm{\infty}{0} = 0$.
  Now given $a,b\in \A^\bd$ and $\lambda, \mu \in {]0,\infty[}$ such that $a^*a \lesssim \lambda^2 \Unit$ and
  $b^*b \lesssim \mu^2 \Unit$, then
  \begin{align*}
    (\lambda + \mu)^2 \Unit - &(a+b)^*(a+b)
    = \\
    &=
    \bigg( 1 + \frac{\mu}{\lambda} \bigg)\big( \lambda^2 \Unit - a^*a \big)
    +
    \bigg( 1 + \frac{\lambda}{\mu} \bigg)\big( \mu^2 \Unit - b^*b \big)
    +
    \bigg( \frac{\sqrt{\lambda}}{\sqrt{\mu}} b - \frac{\sqrt{\mu}}{\sqrt{\lambda}} a \bigg)^*
    \bigg( \frac{\sqrt{\lambda}}{\sqrt{\mu}} b - \frac{\sqrt{\mu}}{\sqrt{\lambda}} a \bigg)
  \end{align*}
  is positive, so $\seminorm{\infty}{a+b} \le \lambda + \mu$. Moreover,
  $(ab)^*(ab) = b^*a^*a\,b \lesssim \lambda^2 b^*b \lesssim (\lambda \mu)^2 \Unit$ 
  shows that $\seminorm{\infty}{ab} \le \lambda \mu$.
  Thus $\seminorm{\infty}{a+b} \le \seminorm{\infty}{a} + \seminorm{\infty}{b}$
  and $\seminorm{\infty}{ab} \le \seminorm{\infty}{a} \seminorm{\infty}{b}$,
  and especially $a+b, ab \in \A^\bd$. So $\A^\bd$
  is a unital subalgebra of $\A$ and $\seminorm{\infty}{\argument}$ a submultiplicative
  seminorm on $\A^\bd$.
  
  In order to show that $\A^\bd$ is stable under $\argument^*$, let $a\in \A^\bd$ be given as well as $\lambda \in {]0,\infty[}$
  such that $a^*a \lesssim \lambda^2 \Unit$. Then
  \begin{equation*}
    0
    \lesssim
    \frac{ (\lambda^2 \Unit - aa^* )^2 }{ \lambda^2 }
    =
    \lambda^2 \Unit - 2 aa^* + a\,\frac{a^*a}{\lambda^2}\,a^*
    \lesssim
    \lambda^2 \Unit - aa^*
  \end{equation*}
  shows that $\seminorm{\infty}{a^*} \le \lambda$. It follows that $\A^\bd$ is a unital $^*$\=/subalgebra
  of $\A$ and $\seminorm{\infty}{a^*} = \seminorm{\infty}{a}$ for all $a\in \A$ because $\argument^*$
  is an involution. 
  
  The $C^*$-property also is fulfilled: The increasing and continuous map 
  $\argument^2\colon {[0,\infty[}\to {[0,\infty[}$ commutes with infima so that
  Definition~\ref{definition:inftyseminorm}
  yields $\seminorm{\infty}{a}^2 = \inf \set[\big]{\lambda^2 \in {]0,\infty[}}{a^*a \lesssim \lambda^2 \Unit}$
  for all $a\in \A^\bd$, which coincides with $\seminorm{\infty}{a^*a}$ by Proposition~\ref{proposition:inftyseminorm}.
  So $\seminorm{\infty}{\argument}$ is indeed a $C^*$\=/seminorm on $\A^\bd$ and its kernel $\mathcal{K}$
  is a $^*$\=/ideal of $\A^\bd$. % Bis hierhin könnte raus
  
  Now given $a, b \in \A$ such that $\seminorm{\infty}{a} = 0$, 
  then $(ab)^*(ab) = b^*a^*a\,b \lesssim \epsilon^2 b^*b$ holds for all $\epsilon \in {]0,\infty[}$,
  which implies $(ab)^*(ab) \lesssim 0$ if $\A$ is additionally Archimedean. In this case
  $\seminorm{\infty}{ab} = 0$ so that $\mathcal{K}$ is even a $^*$\=/ideal of $\A$.
  
  Finally, assume that $\A$ is even an Archimedean ordered $^*$\=/algebra and let 
  $a\in \mathcal{K}$ be given. If $a$ is Hermitian, then $a=0$ by Proposition~\ref{proposition:inftyseminormIsmin}.
  Otherwise $a$ can be expressed as the linear combination $a = \RE(a) + \I \IM(a)$
  of Hermitian elements $\RE(a), \IM(a) \in \mathcal{K}$, which are both $0$ so that again $a=0$.
\end{proof}
So we see that uniformly bounded Archimedean ordered $^*$\=/algebras with the norm 
$\seminorm{\infty}{\argument}$ are pre-$C^*$\=/algebras 
(i.e. $^*$\=/algebras endowed with a $C^*$-norm). Using some standard results about $C^*$\=/algebras,
e.g. the possibility to represent every $C^*$\=/algebra as a $^*$\=/algebra of bounded operators
on a Hilbert space by the Gelfand–Naimark theorem, one can also show that the converse is true as well: 
every pre-$C^*$\=/algebra with the canonical order inherited from its completion to a $C^*$\=/algebra 
is a uniformly bounded Archimedean ordered $^*$\=/algebra. It will be interesting to extend 
the concept of completeness of a $C^*$\=/algebra to general Archimedean ordered $^*$\=/algebras.
While $\seminorm{\infty}{\argument}$ is finite only on the uniformly bounded
elements, and thus does not describe a norm on all Archimedean ordered $^*$\=/algebras,
it still allows to construct a translation-invariant metric:
\begin{definition}
  Let $\A$ be an Archimedean ordered $^*$\=/algebra, then the \neu{uniform metric}
  on $\A$ is defined as the map $\metric_\infty\colon \A\times \A \to {[0,\infty[}$,
  \begin{equation}
    (a,b) \mapsto \metric_\infty(a,b) \coloneqq \min\big\{ \seminorm{\infty}{a-b}, 1 \big\}
    .
  \end{equation}
  All metric notions will always refer to this uniform metric, and $\A$ is especially
  called \neu{uniformly complete} if it is complete with respect to $\metric_\infty$. 
\end{definition}
Note that it is easy to check that $\metric_\infty$ is indeed a translation-invariant metric.

In this language, $C^*$\=/algebras are the uniformly bounded and uniformly complete Archimedean ordered $^*$\=/algebras.
However, neither the product, nor the left or right multiplication with a fixed element are
continuous in the general case: Consider the $^*$\=/algebra $\CC[x]$ of polynomials in one Hermitian
element $x$ like in Example~\ref{example:poly} with the $\RR$-pointwise ordering. Then 
$\lim_{n\to \infty} \Unit / n = 0$ but the sequence
$\NN \ni n \mapsto x / n \in \CC[x]_\Hermitian$ does not converge.
Nevertheless, this metric is still sufficiently well-behaved for some purposes. For example,
it is easy to see that every positive unital $^*$\=/homomorphism between Archimedean
ordered $^*$\=/algebras is automatically continuous with respect to the uniform metric,
because this metric is induced by the order. Moreover:
\begin{lemma} \label{lemma:csersatz}
  Let $\A$ be a quasi-ordered $^*$\=/algebra and $a,b\in \A$, then
  \begin{equation}
    a^* b + b^*a \lesssim \chi^{-2} a^*a + \chi^2 b^*b
  \end{equation}
  holds for all $\chi\in {]0,\infty[}$.
\end{lemma}
\begin{proof}
  This follows from $0 \lesssim (\chi^{-1} a - \chi b)^*(\chi^{-1} a - \chi b) = \chi^{-2} a^*a - a^*b - b^*a + \chi^2 b^*b$.
\end{proof}
\begin{proposition} \label{proposition:closedstuff}
  Let $\A$ be an Archimedean ordered $^*$\=/algebra and $S\subseteq \A_\Hermitian$, then 
  the space $\A_\Hermitian$ of Hermitian elements, the space $\A^\bd$ of uniformly bounded
  elements, the commutant $S'$, the bicommutant $S''$ and the sets $\set{a\in \A_\Hermitian}{a\le s\text{ for all }s\in S}$
  and $\set{a\in \A_\Hermitian}{a\ge s\text{ for all }s\in S}$ are closed in $\A$
  with respect to the uniform metric.
\end{proposition}
\begin{proof}
  Using that $\seminorm{\infty}{\RE(a)} \le \seminorm{\infty}{a}$ and
  $\seminorm{\infty}{\IM(a)} \le \seminorm{\infty}{a}$ hold for all $a\in \A^\bd$,
  it is easy to check that the $\RR$\=/linear projectors $\RE,\IM\colon \A \to \A$ are continuous
  and thus $\A_\Hermitian = \IM^{-1}(\{0\})$ is closed.
  
  Now consider a sequence $(a_n)_{n\in \NN}$ in $\A_\Hermitian$ that converges against some 
  $\hat{a} \coloneqq \lim_{n\to \infty}a_n \in \A_\Hermitian$ and let
  $\epsilon \in {]0,\infty[}$ be given, then there exists an $n\in \NN$ such that
  $\seminorm{\infty}{\hat{a}-a_n} \le \epsilon$, i.e.\ 
  $-\epsilon \Unit \le \hat{a} - a_n \le \epsilon \Unit$ by Proposition~\ref{proposition:inftyseminormIsmin}.
  If all $a_n$ with $n\in \NN$ are uniformly bounded, then this shows that $\hat{a}$ is also uniformly bounded,
  so $\A^\bd \cap \A_\Hermitian$ and $\A^\bd = \RE^{-1}(\A^\bd \cap \A_\Hermitian) \cap \IM^{-1}(\A^\bd \cap \A_\Hermitian)$
  are closed in $\A$.
  
  Moreover, let $s\in \A_\Hermitian$ be given. If $a_n \le s$ for $n\in \NN$, then
  $\hat{a} \le a_n + \epsilon \Unit \le s + \epsilon \Unit$, and if $a_n \ge s$ for $n\in \NN$,
  then $\hat{a} \ge a_n - \epsilon \Unit \ge s - \epsilon \Unit$. If $a_n \in \{s\}'$ for $n\in \NN$, then 
  \begin{equation*}
    \I (\hat{a} s - s\hat{a}) 
    = 
    (\hat{a}-a_n) (\I s) + (-\I s)(\hat{a}-a_n)
    \le 
    \epsilon^{-1} (\hat{a}-a_n)^2 + \epsilon s^2 
    \le 
    \epsilon (\Unit + s^2)
  \end{equation*}
  by Lemma~\ref{lemma:csersatz} with $\chi = \sqrt{\epsilon}$ and
  Lemma~\ref{lemma:ordersquare}. As $\{-s\}' = \{s\}'$, the same
  estimate holds with $-s$ in place of $s$, so
  $-\epsilon(\Unit + s^2) \le \I (\hat{a} s - s\hat{a}) \le \epsilon(\Unit+s^2)$.
  
  Using that $\A$ is Archimedean, these estimates show that 
  $\set{a\in \A_\Hermitian}{a\le s}$, $\set{a\in \A_\Hermitian}{a\ge s}$, 
  and $\{s\}' \cap \A_\Hermitian$ are closed in $\A$. As 
  $\{s\}' = \RE^{-1}(\{s\}' \cap \A_\Hermitian) \cap \IM^{-1}(\{s\}' \cap \A_\Hermitian)$
  also $\{s\}'$ is closed in $\A$. Consequently, intersections
  of such sets, and especially $\set{a\in \A_\Hermitian}{a\le s\text{ for all }s\in S}$,
  $\set{a\in \A_\Hermitian}{a\ge s\text{ for all }s\in S}$ and $S'$ are closed.
  From $S\subseteq \A_\Hermitian$ it follows that $S'$ is stable under $\argument^*$,
  and thus $S'' = (S'\cap \A_\Hermitian)'$ is also closed in $\A$.
\end{proof}
\section{Radical and Symmetric Ordered \texorpdfstring{$^*$-Algebras}{*-Algebras}} \label{sec:rad}
The only compatibility between order and multiplication that has been discussed so far is the axiom
of quasi-ordered $^*$\=/algebras $\A$ that $b^*a\,b \in \A_\Hermitian^+$ for all $a\in \A_\Hermitian^+$
and all $b\in \A$. If the order is sufficiently nice (especially antisymmetric and Archimedean),
then this is indeed enough to guarantee that the elements of $\A$ essentially behave like
adjointable endomorphisms on a pre-Hilbert space, which can be made rigorous by a representation
theorem like in \cite{schoetz:preprintGelfandNaimarkTheorems}. However, it is well-known
that such $^*$\=/algebras of (unbounded) adjointable endomorphisms can still exhibit some unexpected
behaviour because Hermitian endomorphisms need not be essentially self-adjoint.
Because of this, it will be necessary to introduce another compatibility 
between order and multiplication that gurantees that commuting elements essentially
behave like complex-valued functions (see again \cite{schoetz:preprintGelfandNaimarkTheorems}):
\begin{definition} \label{definition:radical}
  Let $\A$ be an ordered $^*$\=/algebra, then an element $a\in \A_\Hermitian$ is called
  \neu{coercive} if there exists an $\epsilon \in {]0,\infty[}$ such that $a \ge \epsilon \Unit$.
  An ordered $^*$\=/algebra $\A$ is called \neu{radical} if the following
  is fulfilled: Whenever $a,b\in \A_\Hermitian$ are two commuting elements such that $a$ is coercive
  and $ab \ge 0$, then $b\ge 0$.
\end{definition}
One obvious example of radical Archimedean ordered $^*$\=/algebras are function algebras
like in Example~\ref{example:fun}. Non-commutative examples will be constructed later on.
Some basic observations about radical Archimedean ordered $^*$\=/algebras are:
\begin{proposition} \label{proposition:radicalOrderSquare}
  Let $\A$ be a radical Archimedean ordered $^*$\=/algebra, $a,b\in \A_\Hermitian$ commuting
  and $a\ge 0$. Then $b^2 \le a^2$ is equivalent to $-a \le b \le a$.
\end{proposition}
\begin{proof}
  One argues like in the proof of Lemma~\ref{lemma:ordersquare}:
  First assume that $b^2 \le a^2$, then also $b^2 \le (a+\epsilon \Unit)^2$ for all $\epsilon \in {]0,\infty[}$,
  so
  \begin{equation*}
    2(a+\epsilon \Unit)(a+\epsilon \Unit \pm b) 
    =
    (a+\epsilon \Unit)^2 - b^2 + ( a+ \epsilon \Unit \pm b )^2
    \ge
    0
    .
  \end{equation*}
  As $\A$ is radical, this shows that $a+\epsilon \Unit \pm b \ge 0$, and then $a\pm b\ge 0$
  because $\A$ is also Archimedean; so $-a \le b \le a$.
  Conversely, if $-a \le b \le a$, then also $-(a+\epsilon \Unit) \le b \le a+\epsilon \Unit$ for all $\epsilon \in {]0,1]}$,
  so
  \begin{equation*}
    2(a+\epsilon\Unit)\big((a+\epsilon \Unit)^2 - b^2\big)
    =
    (a+\epsilon\Unit+b)(a+\epsilon\Unit-b)(a+\epsilon\Unit+b)
    +
    (a+\epsilon\Unit-b)(a+\epsilon\Unit+b)(a+\epsilon\Unit-b)
    \ge
    0
    .
  \end{equation*}
  As $\A$ is radical, this shows that $(a+\epsilon \Unit)^2 - b^2 \ge 0$,
  so $b^2 \le (a+\epsilon \Unit)^2 \le a^2 + \epsilon(2a+\Unit)$.
  It follows that $b^2 \le a^2$ because $\A$ is Archimedean.
\end{proof}
\begin{corollary} \label{corollary:radicalproducts}
  If $\A$ is a radical Archimedean ordered $^*$\=/algebra and $a,b\in \A_\Hermitian^+$ commute,
  then $ab \ge 0$.
\end{corollary}
\begin{proof}
  From $-(a+b) \le a-b \le a+b$ we get $(a-b)^2 \le (a+b)^2$, so $4ab = (a+b)^2 - (a-b)^2 \ge 0$.
\end{proof}
Note that even in the matrix $^*$\=/algebra $\CC^{2\times 2}$ with the composition
of complex conjugation and transposition as $^*$\=/involution and the usual order on the 
Hermitian matrices, which is a $C^*$-algebra and certainly should be regarded as one of 
the most well-behaved ordered $^*$\=/algebras, there exist Hermitian (but not commuting) matrices $a$ and $b$
with $0 \le a \le b$ which do not fulfil $a^2 \le b^2$. A standard example is
\begin{equation*}
  a = \binom{\,2\,\,\,\,2\,}{\,2\,\,\,\,2\,}
  \quad\quad\text{and}\quad\quad
  b = \binom{\,6\,\,\,\,0\,}{\,0\,\,\,\,3\,}
  .
\end{equation*}
Because of this, one should not expect an analog of Proposition~\ref{proposition:radicalOrderSquare}
to be fulfilled for non-commutating Hermitian elements in well-behaved examples.
Using results from real algebraic geometry, the above Corollary~\ref{corollary:radicalproducts} can
be improved significantly, which demonstrates the importance of the radical-property:
\begin{proposition} \label{proposition:polynomialcalculus}
  Let $\A$ be a radical Archimedean ordered $^*$\=/algebra, $a_1,\dots,a_N \in \A_\Hermitian$ with $N\in \NN$
  pairwise commuting and $p_1,\dots,p_M \in \RR[x_1,\dots,x_N] \cong \CC[x_1,\dots,x_N]_\Hermitian$ with $M\in \NN$
  polynomials fulfilling
  $p_m(a_1,\dots,a_N) \in \A_\Hermitian^+$ for all $m\in \{1,\dots,M\}$. Define the associated semialgebraic set
  \begin{equation}
    S \coloneqq \set[\big]{s\in \RR^N}{p_m(s_1,\dots,s_N) \ge 0 \text{ for all }m\in \{1,\dots,M\}}
    .
  \end{equation}
  Then the unital $^*$\=/homomorphism $\CC[x_1,\dots,x_N] \ni q \mapsto q(a_1,\dots,a_N) \in \A$ is positive
  with respect to the $S$-pointwise order on $\CC[x_1,\dots,x_N]$.
\end{proposition}
\begin{proof}
  First let $q \in \CC[x_1,\dots,x_N]_\Hermitian$ be given such that $q(s_1,\dots,s_N) > 0$ for all $s \in S$.
  By the Positivstellensatz of Krivine and Stengle, there exist two polynomials
  \begin{equation*}
    r_1,r_2 \in \genPos[\Big]{\set[\Big]{\prod\nolimits_{m=1}^M (p_m)^{\sigma(m)}}{\sigma(1),\dots,\sigma(M) \in \{0,1\}}}
  \end{equation*}
  such that $(1+r_1) q = 1+r_2$ (this version of the Positivstellensatz
  can be obtained from the more traditional formulation $r_1 q = 1+r_2$ by the well-known trick of
  multiplying with $q$ and adding the identities, which yields $(1+r_1+r_2) q = 1+r_2 + r_1 q^2$).
  This shows that $\big(\Unit + r_1(a_1,\dots,a_N)\big) q(a_1,\dots,a_N) = \Unit+r_2(a_1,\dots,a_N)$
  and from the previous Corollary~\ref{corollary:radicalproducts} it follows that $r_1(a_1,\dots,a_N), r_2(a_1,\dots,a_N) \in \A_\Hermitian^+$,
  so $q(a_1,\dots,a_N)\in \A_\Hermitian^+$ because $\A$ is radical.
  
  For general $q \in \CC[x_1,\dots,x_N]_\Hermitian$ which is $S$-pointwise positive, this shows that
  $q(a_1,\dots,a_N) + \epsilon \Unit = (q+\epsilon)(a_1,\dots,a_N) \in \A_\Hermitian^+$
  for all $\epsilon \in {]0,\infty[}$,
  and thus $q(a_1,\dots,a_N) \in \A_\Hermitian^+$ because $\A$ is Archimedean.
\end{proof}
From this proof it also becomes clear that $\CC[x_1,\dots,x_N]$ with the algebraic order is not radical
for $N\ge 2$,
because there exist real polynomials that are (strictly) pointwise positive on whole $\RR^N$
but not sums of squares, hence not algebraically positive. The first paragraph thus fails for this algebra
and $a_n \coloneqq x_n$, $M\coloneqq1$, $p_1 \coloneqq 1$.

In order to construct radical Archimedean ordered $^*$\=/algebras, it will be helpful to
discuss algebras in which many elements are invertible. Recall that a $^*$\=/algebra $\A$ is
called symmetric if $a\pm \I\Unit$ has a multiplicative inverse for all $a\in \A_\Hermitian$,
or equivalently if $\Unit+a^2$ is invertible for all $a\in \A_\Hermitian$. However, there
are also similar, but non-equivalent notions where one demands that e.g.\ 
$\Unit+a^*a$ is invertible for all $a\in \A$ or that $\Unit+ \sum_{n=1}^N a_n^*a_n$
is invertible for all $a_1,\dots,a_N \in \A$ with $N\in \NN$,
see \cite[Chap.~9.8]{palmer:BanachAlgebras2} for a comparison. In ordered $^*$\=/algebras, there is another, even stronger possibility:
\begin{definition}
  An ordered $^*$\=/algebra $\A$ is called \neu{symmetric} if every coercive element of $\A_\Hermitian$
  has a multiplicative inverse.
\end{definition}
In order to prove that every symmetric Archimedean ordered $^*$\=/algebra is radical, we need some 
preliminary lemmas:
\begin{lemma} \label{lemma:inverseordering}
  Let $\A$ be an ordered $^*$\=/algebra, $a\in \A_\Hermitian$ coercive and $\epsilon \in {]0,\infty[}$
  such that $a \ge \epsilon \Unit$, then $a^{-1}$ is Hermitian, positive and uniformly bounded with 
  $\seminorm{\infty}{a^{-1}} \le \epsilon^{-1}$.
\end{lemma}
\begin{proof}
  We have $a^{-1} = (a\,a^{-1})^* a^{-1} = (a^{-1})^* a\, a^{-1} \in \A_\Hermitian^+$, and
  $a = \epsilon^{-1} a^2 - \epsilon^{-1}(a-\epsilon\Unit)^2 - (a -\epsilon \Unit) \le \epsilon^ {-1} a^2$
  implies $a^{-1} = a^{-1} a\,a^{-1} \le \epsilon^{-1} a^{-1} a^2a^{-1} = \epsilon^{-1} \Unit$
  so that $\seminorm{\infty}{a^{-1}} \le \epsilon^{-1}$ by Proposition~\ref{proposition:inftyseminorm}.
\end{proof}
\begin{lemma} \label{lemma:approximateSqrt}
  Let $\A$ be an ordered $^*$\=/algebra and $a\in (\A^\bd)^+_\Hermitian$,
  then there exist two sequences of polynomials $(p_n )_{n\in \NN}$ and 
  $(q_n )_{n\in \NN}$ in $\RR[x] \cong \CC[x]_\Hermitian$
  such that the identity $a +q_n(a) = p_n^2(a)$ and the estimates 
  $0 \le p_n(a)$ and $0 \le q_n(a) \le \Unit / n$ hold.
\end{lemma}
\begin{proof}
  Let $u \coloneqq \seminorm{\infty}{a} + 1$ so that $0 \le a \le u \Unit$ by Proposition~\ref{proposition:inftyseminorm}.
  By the (Stone-)Weierstraß theorem, applied to the continuous function 
  $\sqrt{\argument} \colon [0,u] \to \RR$, there exists for every $n\in \NN$ a polynomial
  $p_n' \in \CC[x]_\Hermitian$ such that $\abs{\sqrt{t} - p_n'(t)} \le 1/(4n(\sqrt{u}+1))$
  holds for all $t\in [0,u]$. Define $p_n \coloneqq p_n'+1/(4n(\sqrt{u}+1))$
  and $q_n \coloneqq p_n^2 - x$. Then the estimate 
  $0 \le \sqrt{t} \le p_n(t) \le \sqrt{t}+1/(2n(\sqrt{u}+1))$
  and thus $0 \le p_n^2(t)-t \le 1/n$ hold for all $t\in[0,u]$.
  
  Using the fundamental theorem of algebra, one can show that
  every polynomial $r \in \CC[x]_\Hermitian$ which is pointwise positive on $[0,u]$
  is an element of $\genPos{\{x, u-x\}}$, see e.g.\ 
  \cite[Prop.~3.3]{schmuedgen:TheMomentProblem},
  hence $r(a) \in \algebra{A}^+_\Hermitian$.
  The pointwise estimates for $p_n$ and $q_n$ thus
  imply that $0 \le p_n(a)$ and $0 \le q_n(a) \le \Unit / n$,
  and the identity $a + q_n(a) = p_n(a)^2$ is fulfilled by construction.
\end{proof}
\begin{proposition} \label{proposition:symisrad}
  Every symmetric Archimedean ordered $^*$\=/algebra is radical.
\end{proposition}
\begin{proof}
  Let two commuting elements $a,b\in \A_\Hermitian$ and $\epsilon \in {]0,\infty[}$ 
  be given such that $a$ is coercive with $a \ge \epsilon \Unit$ and $ab \ge 0$.
  Using the previous Lemmas~\ref{lemma:inverseordering} and \ref{lemma:approximateSqrt}
  one can construct sequences of polynomials $(p_n)_{n\in \NN}$ and $(q_n)_{n\in \NN}$
  such that $a^{-1} + q_n(a^{-1}) = p_n(a^{-1})^2$ with $0 \le q_n(a^{-1}) \le \Unit/n$
  for all $n\in \NN$, so $0 \le p_n(a^{-1}) \,a b\, p_n(a^{-1}) = b + q_n(a^{-1})\, a b$.
  Using Lemma~\ref{lemma:csersatz} with $\chi = \sqrt{n}$ and that $q_n(a^{-1})^2 \le \Unit/n^2$
  by Lemma~\ref{lemma:ordersquare}, it follows that 
  $2\, q_n(a^{-1})\, a b \le \chi^{-2} a^2 b^2 + \chi^2 q_n(a^{-1})^2 \le (a^2 b^2 + \Unit)/n$.
  As $\A$ is Archimedean it follows that $0 \le b$, so $\A$ is radical.
\end{proof}
In the uniformly complete case, we will also see that the various notions of symmetric $^*$\=/algebras 
that were mentioned before are actually equivalent:
\begin{lemma} \label{lemma:csersatz2}
  Let $\A$ be an ordered $^*$\=/algebra, $a,b\in \A$ and $d \in \A_\Hermitian^+$,
  then
  \begin{equation}
    a^* c \,b + b^* c\, a \le  a^*d\,a + b^*d\, b
    \label{eq:csersatz2}
  \end{equation}
  holds for all $c \in \A_\Hermitian$ fulfilling $-d \le c \le d$.
\end{lemma}
\begin{proof}
  Given $c \in \A_\Hermitian$ with $-d \le c \le d$, then write $c_{(+)} \coloneqq (d+c)/2 \in \A_\Hermitian^+$ and $c_{(-)} \coloneqq (d-c)/2\in \A_\Hermitian^+$.
  Note that $c = c_{(+)}-c_{(-)}$ and $d = c_{(+)} + c_{(-)}$.
  From $0 \le (a - b)^* c_{(+)}(a - b)$ and $0 \le (a + b)^* c_{(-)}(a + b)$
  it follows that
  \begin{equation*}
    a^* c_{(+)}b + b^* c_{(+)}a 
    \le
    a^* c_{(+)}a + b^* c_{(+)}b
  \quad\quad\text{and}\quad\quad
    -a^* c_{(-)}b - b^* c_{(-)}a 
    \le
    a^* c_{(-)}a + b^* c_{(-)}b
  \end{equation*}
  hold. Adding these two estimates yields \eqref{eq:csersatz2}.
\end{proof}
\begin{lemma} \label{lemma:invertibleWhenComplete}
  Let $\A$ be an Archimedean ordered $^*$\=/algebra, $\hat{a} \in \A_\Hermitian$
  and $(a_n)_{n\in \NN}$ a sequence in $\A_\Hermitian$ of invertible elements
  such that the sequence of their inverses converges with respect to the uniform metric
  against some limit $e \coloneqq \lim_{n\to \infty} a_n^{-1} \in \A$. Moreover, assume
  that there exist elements $c,d \in \A_\Hermitian^+$ such that $a_n^2 \le c$ for all $n\in \NN$
  and such that for all $\epsilon \in {]0,\infty[}$ there exists an $N \in \NN$
  for which $-\epsilon d \le \hat{a} - a_n \le \epsilon d$ is fulfilled for all $n\in \NN$ with $n\ge N$.
  Then $\hat{a}$ is invertible and $\hat{a}^{-1} = e = \lim_{n\to \infty} a_n^{-1}$.
\end{lemma}
\begin{proof}
  As all $a_n$ with $n\in \NN$ and thus also their inverses $a_n^{-1}$ are Hermitian, $e$ is Hermitian by Proposition~\ref{proposition:closedstuff}.
  Therefore it is sufficient to show that
  $\hat{a}e = \Unit$, which then also implies $e\hat{a} = (\hat{a} e)^* = \Unit$.
  
  So let $\epsilon \in {]0,\infty[}$ be given, then there exists an $n\in \NN$ such that 
  $-\epsilon d \le \hat{a} - a_n \le \epsilon d$ and $\seminorm{\infty}{e-a_n^{-1}} \le \epsilon$ hold,
  thus also $(e-a_n^{-1})^2 \le \epsilon^2 \Unit$ by Proposition~\ref{proposition:inftyseminormIsmin}.
  Using the previous Lemma~\ref{lemma:csersatz2} and Lemma~\ref{lemma:csersatz}
  with $\chi \coloneqq 1 / \sqrt{\epsilon}$ one finds that
  \begin{align*}
    \underbrace{
      (\I^k \Unit)^* (\hat{a}-a_n) e + e (\hat{a}-a_n) (\I^k \Unit)
    }_{
      \le \epsilon(d+ede)
    }
    +
    \underbrace{
      (\I^k a_n)^* (e-a_n^{-1}) + (e-a_n^{-1}) (\I^k a_n)
    }_{
      \le \epsilon(c + \Unit)
    }
    \le 
    \epsilon (\Unit+c+d+ede)
  \end{align*}
  holds for all $k\in \{0,1,2,3\}$, or equivalently, $-\epsilon (\Unit+c+d+ede) \le 2\RE(\hat{a}e - \Unit) \le \epsilon (\Unit+c+d+ede)$
  and $-\epsilon (\Unit+c+d+ede) \le 2\IM(\hat{a}e - \Unit) \le \epsilon (\Unit+c+d+ede)$
  because $\hat{a}e - \Unit = (\hat{a} - a_n) e + a_n(e-a_n^{-1})$.
\end{proof}
\begin{proposition} \label{proposition:invertibleWhenComplete}
  Let $\A$ be a uniformly complete Archimedean ordered $^*$\=/algebra and $a,b\in\A^+_\Hermitian$ commuting
  such that $a$ is coercive, $b$ invertible and $a \le b^2$. Then $a$ is also invertible.
\end{proposition}
\begin{proof}
  It is sufficient to show that $\hat{a} \coloneqq bab$ has an inverse, then
  $a$ is also invertible with $a^{-1} = b\hat{a}^{-1}b$.
  There is an $\epsilon \in {]0,\infty[}$ such that $\epsilon \Unit \le a$,
  and consequently $\epsilon \Unit \le b^2$ and $\epsilon^2 \Unit \le \epsilon b^2 \le \hat{a} \le b^4$ hold.
  
  Define $a_n \coloneqq \hat{a} + b^4/n$ for all $n\in \NN$,
  $c \coloneqq \hat{a}^2 + 3b^8$ and $d \coloneqq b^4$, then
  $a_n^2 = \hat{a}^2 + 2 b^2\hat{a} b^2 /n + b^8/n^2 \le c$
  and $-d/n = \hat{a} - a_n \le d / n$.
  In order to apply the previous
  Lemma~\ref{lemma:invertibleWhenComplete} it only remains to show that
  all $a_n$ with $n\in \NN$ are invertible and that the sequence of their inverses
  is a Cauchy sequence.
  
  Consider $b^{-2} a_n b^{-2} = b^{-2} \hat{a} b^{-2} + \Unit/n$.
  Then $\Unit / n \le b^{-2} a_n b^{-2} \le (1+1/n)\Unit$, 
  so $b^{-2} a_n b^{-2}$ is a coercive element of $\A^\bd$. From Proposition~\ref{proposition:closedstuff}
  it follows that $\A^\bd$ is uniformly complete itself, hence a $C^*$\=/algebra, 
  so $b^{-2} a_n b^{-2}$ is invertible in $\A^\bd$
  (the inverse can be constructed explicitly e.g.\ using a Neumann series).
  Consequently, $a_n$ is also invertible with $a_n^{-1} = b^{-2} (b^{-2} a_n b^{-2})^{-1} b^{-2}$.
  Moreover, using $- \abs{m^{-1}-n^{-1}} b^4 \le a_n-a_m \le \abs{m^{-1}-n^{-1}} b^4$
  and Lemma~\ref{lemma:csersatz2}, one obtains the estimate
  \begin{equation*}
    a_m^{-1} - a_n^{-1}
    =
    \frac{
      a_m^{-1}(a_n - a_m)a_n^{-1}
      +
      a_n^{-1}(a_n - a_m)a_m^{-1}
    }{
      2
    }
    \le
    \abs[\bigg]{\frac{1}{m}-\frac{1}{n}}
    \frac{
      a_m^{-1} b^4 a_m^{-1}
      +
      a_n^{-1} b^4 a_n^{-1}
    }{
      2
    }
  \end{equation*}
  for all $m,n\in \NN$. From $\epsilon b^2 \le \hat{a}$
  it follows that $\epsilon \Unit \le b^{-1}\hat{a}b^{-1}$
  and thus 
  \begin{equation*}
    \epsilon^2 b^4 
    \le 
    (\hat{a}-\epsilon b^2)^2 + \epsilon^2 b^4
    =
    \hat{a}^2 -2 \epsilon b(\hat{a} - \epsilon b^2)b 
    \le 
    \hat{a}^2
    \le 
    \hat{a}^2 + 2 b^2 \hat{a} b^2 / n + b^8 / n^2
    = 
    a_n^2
  \end{equation*}
  for all $n\in \NN$.
  The combination of these estimates yields
  $a_m^{-1} - a_n^{-1} \le \epsilon^{-2} \abs{m^{-1}-n^{-1}} \Unit$
  for all $m,n\in \NN$, so $(a_n^{-1})_{n\in \NN}$ is indeed a Cauchy sequence.
\end{proof}
\begin{corollary} \label{corollary:symcond}
  Let $\A$ be a uniformly complete Archimedean ordered $^*$\=/algebra in which $\Unit+a^2$ is
  invertible for all $a\in \A_\Hermitian$, then $\A$ is symmetric.
\end{corollary}
\begin{proof}
  Given a coercive $a \in \A_\Hermitian$, then one can apply the previous Proposition~\ref{proposition:invertibleWhenComplete}
  with $b \coloneqq \Unit + a^2$ because $a \le 2a = \Unit+a^2 - (\Unit -a)^2 \le \Unit+a^2 \le (\Unit + a^2)^2$.
\end{proof}
\section{\texorpdfstring{$\Phi^*$-Algebras}{Phi*-Algebras}} \label{sec:supsetc}
It has already been mentioned in Section~\ref{sec:preliminaries} that Riesz spaces which carry a
non-commutative multiplication have rather pathological properties. Because of this, a
well-behaved non-commutative generalization of the notion of $\Phi$\=/algebras must necessarily
deal with some restrictions to the infima and suprema. Moreover, like in $\Phi$\=/algebras, there
should also be a compatibility between suprema, infima and the product,
but it might not be immediately clear what exactly this compatibility should be. The following
observation, which gives a mostly algebraic characterization of suprema and infima,
might serve as a motivation (recall that $\argument'$ denotes the commutant):
\begin{proposition} \label{proposition:supinfchar}
  Let $\A$ be a radical Archimedean ordered $^*$\=/algebra, $a,b\in \A_\Hermitian$
  commuting and let $x \in \{a,b\}'' \cap \A_\Hermitian$ be such that $x^2 + ab = x(a+b)$, then the following
  holds:
  \begin{itemize}
    \item If $2x \ge a+b$, then $x$ is the supremum of $a$ and $b$ in $\{a,b\}' \cap \A_\Hermitian$.
    \item If $2x \le a+b$, then $x$ is the infimum of $a$ and $b$ in $\{a,b\}' \cap \A_\Hermitian$.
  \end{itemize}
\end{proposition}
\begin{proof}
  First assume that $2x \ge a+b$, then
  $(2x-a-b)^2 = 4x^2 - 4x(a+b) + 4ab +(a-b)^2 = (a-b)^2$
  implies $-(2x-a-b) \le a-b \le 2x-a-b$ by Proposition~\ref{proposition:radicalOrderSquare},
  so $x \ge a$ and $x\ge b$. Moreover, if some $y\in \{a,b\}' \cap \A_\Hermitian$
  also fulfils $y\ge a$ and $y\ge b$, then $0 \le (y-a)(y-b) = y^2 - y(a+b) + ab$
  by Corollary~\ref{corollary:radicalproducts} and thus
  $(2y-a-b)^2 = 4y^2 - 4y(a+b) + 4ab +(a-b)^2 \ge (a-b)^2 = (2x-a-b)^2$.
  Proposition~\ref{proposition:radicalOrderSquare} now shows that
  $2y-a-b \ge 2x-a-b$, so $y\ge x$ and $x$ is indeed the
  supremum of $a$ and $b$ in $\{a,b\}' \cap \A_\Hermitian$.
  If $2x \le a+b$, then one can apply the above argument to $-x$, $-a$ and $-b$.
\end{proof}
The following definition thus makes sense and describes suprema and infima that
fulfil additional algebraic conditions:
\begin{definition} \label{definition:veewedge}
  Let $\A$ be a radical Archimedean ordered $^*$\=/algebra and $a,b\in \A_\Hermitian$ commuting.
  Then $a\vee b$ is (if it exists) the element in $\{a,b\}'' \cap \A_\Hermitian$
  which fulfils
  \begin{equation}
    2(a\vee b) \ge a+b
    \quad\quad \text{and} \quad\quad
    (a\vee b)^2  + ab = (a\vee b)(a+b) 
    .
  \end{equation}
  Similarly, $a\wedge b$ is (if it exists) the element in $\{a,b\}'' \cap \A_\Hermitian$
  which fulfils 
  \begin{equation}
    2(a\wedge b) \le a+b
    \quad\quad \text{and} \quad\quad
    (a\wedge b)^2  + ab = (a\wedge b)(a+b) 
    .
  \end{equation}
  A radical Archimedean ordered $^*$\=/algebra $\A$ in which $a\vee b$ and $a\wedge b$
  exist for all commuting $a,b\in \A_\Hermitian$ will be called a \neu{$\Phi^*$-algebra}.
\end{definition}
Proposition~\ref{proposition:supinfchar} especially guarantees that $a\vee b$ and $a\wedge b$
(if they exist) are uniquely determined as certain suprema and infima.
In a $\Phi^*$\=/algebra $\A$, Proposition~\ref{proposition:supinfchar} also
has a rather trivial, but noteworthy converse: If $a,b \in \A_\Hermitian$
commute and $x \in \{a,b\}'\cap \A_\Hermitian$ is the supremum or infimum
of $a$ and $b$ in some real linear subspace $V$ of $\A_\Hermitian$ such
that $\{a,b\}'' \cap \A_\Hermitian \subseteq V \subseteq \{a,b\}' \cap \A_\Hermitian$,
then $x$ coincides with $a \vee b$ or $a \wedge b$, respectively, due to the uniqueness
of the suprema and infima; so especially $x \in \{a,b\}'' \cap \A_\Hermitian$
and $x^2 + ab = x(a+b)$.

There are some basic results about these suprema and infima which are not very surprising as
they mostly mimic the rules in ordered vector spaces, and which can easily be checked:
\begin{proposition}\label{proposition:veewedge}
  Let $\A$ be a radical Archimedean ordered $^*$\=/algebra and $a,b \in \A_\Hermitian$
  commuting. Then $a\vee b$ exists if and only if $a\wedge b$ exists, and the two are
  related by
  \begin{equation}
    (a\vee b) + (a\wedge b) = a+b
    \quad\quad\text{and}\quad\quad
    (a\vee b) (a\wedge b) = ab
    .
  \end{equation}
  Moreover, if one, hence both of $a\vee b$ and $a\wedge b$ exist, then the following holds:
  \begin{enumerate}
    \item $b\vee a = a\vee b$ and $b \wedge a = a \wedge b$ exist. \label{item:veewedge:kom}
    \item $(\lambda a)\vee(\lambda b) = \lambda (a\vee b)$ and $(\lambda a)\wedge(\lambda b) = \lambda (a\wedge b)$ exist
      for all $\lambda \in {[0,\infty[}$. \label{item:veewedge:mult}
    \item $(- a)\wedge( - b) = -(a\vee b)$ and $(- a)\vee( - b) = -(a\wedge b)$
      exist. \label{item:veewedge:neg}
    \item $(a+c)\vee(b+c) = (a\vee b) + c$ exists for all $c \in \{a,b\}' \cap \A_\Hermitian$ with
      $(a\vee b) + c \in \{a+c,b+c\}''$, and
      $(a+c)\wedge(b+c) = (a\wedge b) + c$ exists for all $c \in \{a,b\}' \cap \A_\Hermitian$ with
      $(a\wedge b) + c \in \{a+c,b+c\}''$. \label{item:veewedge:add}
  \end{enumerate}
\end{proposition}
\begin{proof}
  If $x \in \{a,b\}'' \cap \A_\Hermitian$ fulfils $x^2 + ab = x(a+b)$, then also
  $y \coloneqq a+b-x \in \{a,b\}'' \cap \A_\Hermitian$ fulfils $y^2 + ab = y(a+b)$.
  Especially using $x = a\vee b$ and $x = a\wedge b$ it follows that $a\vee b$ exists
  if and only if $a\wedge b$ exists and that $(a\vee b) + (a\wedge b) = a+b$.
  As a consequence, $(a\vee b)(a\wedge b) = (a\vee b)(a+b)-(a\vee b)^2 = ab$.
  Checking that \refitem{item:veewedge:kom}, \refitem{item:veewedge:mult} and \refitem{item:veewedge:neg}
  hold is easy and for part \refitem{item:veewedge:add} one essentially only needs
  to verify that
  \begin{align*}
    \big((a\vee b) + c\big)^2 + (a+c)(b+c) 
    &= 
    (a\vee b)^2 + ab + 2c (a\vee b) + (a+b+2c)c \\
    &= 
    (a+b)(a\vee b) + 2c (a\vee b) + (a+b+2c)c \\
    &=
    (a+b+2c)\big((a\vee b) + c\big)
    .
  \end{align*}
\end{proof}
With respect to part \refitem{item:veewedge:add} we note that the conditions
$(a\vee b) + c, (a\wedge b) + c \in \{a+c,b+c\}''$ are superfluous
if it is a priori known that $(a+c)\vee(b+c)$ and $(a+c)\wedge(b+c)$ exist, i.e.\ especially
if $\A$ is a $\Phi^*$\=/algebra. This is due to the observation that $(a\vee b) + c$ and
$(a\wedge b) + c$ with $c\in \{a,b\}' \cap \A_\Hermitian$ are the supremum and infimum,
respectively, of $a+c$ and $b+c$ in $V \coloneqq \{a,b\}'\cap \{a+c,b+c\}' \cap \A_\Hermitian$,
so the discussion under Definition~\ref{definition:veewedge} applies.

One important special case of these suprema in ordered $^*$\=/algebras are absolute values:
\begin{definition} \label{definition:abs}
  Let $\A$ be a radical Archimedean ordered $^*$\=/algebra and $a\in \A_\Hermitian$, then the \neu{absolute value}
  of $a$ is defined (if it exists) as the element $\abs{a} \coloneqq a \vee (-a)$. If the absolute value
  exists, then one also defines the \neu{positive part} $a_+ \coloneqq \frac{1}{2}(\abs{a}+a)$ and 
  the \neu{negative part} $a_- \coloneqq \frac{1}{2}(\abs{a}-a)$ of $a$.
\end{definition}
Clearly, $\abs{-a} = \abs{a}$ if $\abs{a}$ exists. By definition, the absolute value of an element 
$a\in \A_\Hermitian$ is (if it exists) the element $\abs{a} \in \{a\}''\cap \A_\Hermitian^+$ that fulfils $\abs{a}^2 = a^2$.
The earlier results about suprema and infima now show that, like for Riesz spaces, the existence of all absolute values
already implies the existence of all suprema and infima of commuting Hermitian elements:
\begin{proposition} \label{proposition:absimpliessup}
  Let $\A$ be a radical Archimedean ordered $^*$\=/algebra and $a,b\in \A_\Hermitian$
  commuting and such that $\abs{a-b}$ exists. Then $a \vee b$ and $a\wedge b$ exist and
  are given by
  \begin{equation}
    a\vee b = \frac{a+b+\abs{a-b}}{2}
    \quad\quad\text{and}\quad\quad
    a\wedge b = \frac{a+b-\abs{a-b}}{2}
    .
    \label{eq:absimpliessup}
  \end{equation}
\end{proposition}
\begin{proof}
  This is just an application of Proposition~\ref{proposition:veewedge} using that
  $\abs{a-b} \in \{a,b\}''$.
\end{proof}
As immediate consequences of Propositions~\ref{proposition:veewedge} and \ref{proposition:absimpliessup} we obtain:
\begin{corollary} \label{corollary:posnegpart}
  Let $\A$ be a radical Archimedean ordered $^*$\=/algebra and $a\in \A_\Hermitian$ such that $\abs{a}$ exists. Then
  $a\vee 0 = a_+ \in A_\Hermitian^+$ and $(-a) \vee 0 = -(a\wedge 0) = a_-\in A_\Hermitian^+$ exist and fulfil $a_++a_- = a$ and $a_-a_+ = a_+a_- = 0$.
\end{corollary}
\begin{corollary} \label{corollary:Phistarchar}
  Let $\A$ be a radical Archimedean ordered $^*$\=/algebra, then $\A$ is a $\Phi^*$\=/algebra
  if and only if $\abs{a}$ exists for all $a\in \A_\Hermitian$.
\end{corollary}
One motivation to study $\Phi^*$\=/algebras is that they indeed are a non-commutative
generalization of $\Phi$\=/algebras:
\begin{proposition} \label{proposition:comphistar}
  Let $\mathcal{R}$ be a $\Phi$\=/algebra, then its complexification 
  $\A \coloneqq \mathcal{R} \otimes \CC$ with $^*$\=/involution and multiplication defined by
  $(r\otimes \lambda)^* \coloneqq r \otimes \cc{\lambda}$
  and $(r\otimes \lambda) (s\otimes \mu) \coloneqq rs \otimes \lambda \mu$
  for all $r,s\in \mathcal{R}$ and all $\lambda, \mu \in \CC$ is a commutative $\Phi^*$\=/algebra.
  Conversely, if $\A$ is a commutative $\Phi^*$\=/algebra, then its real unital subalgebra
  $\A_\Hermitian$ is a $\Phi$\=/algebra.
\end{proposition}
\begin{proof}
  First let $\mathcal{R}$ be a $\Phi$\=/algebra and $\A \coloneqq \mathcal{R} \otimes \CC$.
  Then it is clear that $\A$ is a commutative Archimedean ordered $^*$\=/algebra, and it is also radical:
  Given two commuting $a,b \in \A_\Hermitian \cong \mathcal{R}$
  such that $a \ge \epsilon \Unit$ for some $\epsilon \in {]0,\infty[}$ and $0 \le ab$, then write $b_+ \coloneqq \sup\{b,0\}$ and 
  $b_-\coloneqq \sup\{-b,0\}$. Note that it is not yet clear that $b_+$ and $b_-$
  are the positive and negative part of $b$ like in Definition~\ref{definition:abs}, but it
  follows from the general calculation rules in Riesz spaces and $\Phi$\=/algebras that
  $b = b_+-b_-$ and $b_+b_- = 0$. Consequently, $0 \le b_-ab = -b_- a b_- \le - \epsilon (b_-)^2 \le 0$, so $(b_-)^2 = 0$.
  Proposition~\ref{proposition:nonilpotent} now shows that $b_- = 0$ and therefore $b = b_+ \ge 0$.
  As the order-theoretic absolute value $\abs{a} \coloneqq\sup\{a,-a\} \in \A_\Hermitian^+$
  of any $a\in\A_\Hermitian$ indeed fulfils $\abs{a}^2 = a^2$ by the calculation rules in $\Phi$\=/algebras,
  it also describes the absolute value as in Definition~\ref{definition:abs} and therefore $\A$
  is a $\Phi^*$\=/algebra by the previous Corollary~\ref{corollary:Phistarchar}.
  
  Now let $\A$ be an arbitrary commutative $\Phi^*$\=/algebra. Then $\A_\Hermitian$ is a 
  real commutative unital associative algebra and a Riesz space by Proposition~\ref{proposition:supinfchar}.
  Corollary~\ref{corollary:radicalproducts}
  shows that $ab \in \A_\Hermitian^+$ for all $a,b\in \A_\Hermitian^+$. Given
  $a,b,c\in\A_\Hermitian^+$ with $\inf\{a,b\} = 0$, then $a\wedge b = \inf\{a,b\} = 0$ and thus
  $ab = 0$. It follows that
  $0 \le (ac\wedge b)^2 \le (ac\wedge b)(ac \vee b) = acb = 0$
  holds by Corollary~\ref{corollary:radicalproducts} and Proposition~\ref{proposition:veewedge}, so $(ac\wedge b)^2 = 0$.
  Proposition~\ref{proposition:nonilpotent} now shows that $ac\wedge b = 0$
  and thus $\A_\Hermitian$ is a $\Phi$\=/algebra.
\end{proof}
Non-commutative examples of $\Phi^*$\=/algebras will be described in Sections~\ref{sec:su} and \ref{sec:example}.

Another interesting observation about $\Phi^*$\=/algebras is that injective positive
unital $^*$\=/homomorphisms between them are automatically order embeddings. This is roughly 
similar to the case of $^*$\=/algebras endowed with a Fréchet-topology,
where surjective continuous linear maps are automatically open by the open mapping
theorem:
\begin{proposition} \label{proposition:openmapping}
  Let $\Psi\colon \A \to \algebra{B}$ be an injective positive unital $^*$\=/homomorphism
  from a $\Phi^*$\=/algebra $\A$ to an ordered $^*$\=/algebra $\algebra{B}$,
  then $\Psi$ is automatically an order embedding.
\end{proposition}
\begin{proof}
  Let $a\in \A_\Hermitian$ with $\Psi(a) \ge 0$ be given. Then $a = a_+ - a_-$
  with $a_+,a_- \in \A_\Hermitian^+$ and $a_+ a_- = 0$ by Corollary~\ref{corollary:posnegpart},
  so on the one hand 
  $(a_-)^3 \ge 0$ implies $\Psi(a_-)^3 \ge 0$, and on the other,
  $(a_-)^3 = -\big(a_-aa_-\big)$ implies $\Psi(a_-)^3 = -\Psi(a_-)\,\Psi(a)\,\Psi(a_-) \le 0$,
  so $\Psi(a_-)^3 = 0$. As $\Psi$ is injective, it follows that $(a_-)^3 = 0$ and thus $a_- = 0$
  by Proposition~\ref{proposition:nonilpotent}. So $a = a_+ \ge 0$ and $\Psi$ is an order embedding.
\end{proof}
Moreover, positive unital $^*$\=/homomorphisms between $\Phi^*$\=/algebras
are compatible with $\vee$, $\wedge$ and the absolute value:
\begin{lemma} \label{lemma:sqrtUnique}
  Let $\A$ be a radical Archimedean ordered $^*$\=/algebra and $a\in \A_\Hermitian^+$,
  $b \in \{a\}'' \cap \A_\Hermitian^+$ and $c \in \{a\}' \cap \A_\Hermitian^+$ such that
  $b^2 = c^2$, then $b=c$ and especially $c \in \{a\}''$.
\end{lemma}
\begin{proof}
  Note that $b$ and $c$ commute, so $b^2 = c^2$ implies $b=c$ by Proposition~\ref{proposition:radicalOrderSquare}.
\end{proof}
\begin{proposition} \label{proposition:morphveewedgeabs}
  Let $\Psi \colon \A \to \algebra{B}$ be a positive unital $^*$\=/homomorphism
  between two $\Phi^*$\=/algebras $\A$, $\algebra{B}$ and $a,\tilde{a}\in \A_\Hermitian$ commuting.
  Then 
  \begin{equation*}
    \Psi(a) \vee \Psi(\tilde{a}) = \Psi(a\vee \tilde{a})
    ,\quad\quad
    \Psi(a) \wedge \Psi(\tilde{a}) = \Psi(a\wedge \tilde{a})
    \quad\quad\text{and}\quad\quad
    \abs{\Psi(a)} = \Psi(\abs{a})
  \end{equation*}
  hold.
\end{proposition}
\begin{proof}
  As $\vee$ and $\wedge$ can be expressed using the absolute value, it is enough to
  show that $\abs{\Psi(a)} = \Psi(\abs{a})$ holds. It is easy to check that
  $\Psi(\abs{a}) \in \{\Psi(a)\}' \cap \algebra{B}^+_\Hermitian$ and that
  $\Psi(\abs{a})^2 = \Psi(a)^2$. As it is already known that 
  $\abs{\Psi(a)} \in \{\Psi(a)\}'' \cap \algebra{B}^+_\Hermitian$ exists,
  it follows from the previous Lemma~\ref{lemma:sqrtUnique} that
  $\abs{\Psi(a)} = \Psi(\abs{a})$.
\end{proof}

As a last result we note that in the uniformly complete case, the existence of infima
is helpful for the construction of a multiplicative inverse:
\begin{lemma} \label{lemma:weakOrderUnit}
  Let $\A$ be a radical Archimedean ordered $^*$\=/algebra, $a\in \A_\Hermitian$ and $\lambda\in{]0,\infty[}$.
  If $a\wedge \lambda\Unit$ exists, then it fulfils the estimate $a \le (a\wedge \lambda\Unit) + a^2 / (4\lambda)$.
\end{lemma}
\begin{proof}
  As $a$ and $a\wedge \lambda\Unit$ commute, 
  $a(a\wedge \lambda \Unit) = \big(a(a\wedge \lambda \Unit) + (a\wedge \lambda \Unit)a\big) / 2 \le a^2/4 + (a\wedge \lambda\Unit)^2$
  holds by Lemma~\ref{lemma:csersatz} with $\chi \coloneqq \sqrt{2}$.
  From $(a\wedge \lambda \Unit)^2 + \lambda a = (a + \lambda \Unit) (a\wedge \lambda \Unit)$ it
  now follows that $(a\wedge \lambda \Unit)^2 + \lambda a \le a^2/4 + (a\wedge \lambda\Unit)^2 + \lambda (a\wedge \lambda \Unit)$.
\end{proof}
\begin{lemma} \label{lemma:radicalinverseordering}
  Let $\A$ be a radical Archimedean ordered $^*$\=/algebra and $a,b\in \A^+_\Hermitian$
  with $a\le b$ commuting and invertible, then $a^{-1} \ge b^{-1}$.
\end{lemma}
\begin{proof}
  As $b-a, a^{-1}, b^{-1} \in \{a,b\}'' \cap \A_\Hermitian^+$ are pairwise commuting,
  their product $a^{-1}(b-a)b^{-1} = a^{-1} - b^{-1}$ is positive by Corollary~\ref{corollary:radicalproducts}.
\end{proof}
\begin{proposition} \label{proposition:invconst}
  Let $\A$ be a radical and uniformly complete Archimedean ordered $^*$\=/algebra and let $a\in \A_\Hermitian$
  be a coercive element for which $a\wedge n\Unit$ exists for all $n\in \NN$, then $a$ is invertible.
\end{proposition}
\begin{proof}
  Proposition~\ref{proposition:supinfchar} shows that
  $a \wedge n \Unit$ is, for every $n\in \NN$, the infimum
  of $a$ and $n\Unit$ in the commutative real unital subalgebra
  $\{a\}'' \cap \A_\Hermitian$ of $\A$. As $a$ is coercive, there exists 
  $\epsilon \in {]0,1]}$ such that $\epsilon \Unit \le a$, and then 
  $\epsilon \Unit \le a \wedge n \Unit \le n \Unit$ shows that $a \wedge n \Unit$
  is a coercive element of $\A^\bd$. By Proposition~\ref{proposition:closedstuff},
  $\A^\bd$ is uniformly complete itself, hence a $C^*$\=/algebra, and thus $a \wedge n \Unit$ is invertible.
  
  For fixed $m,n\in \NN$ with $n \le m$, the estimate $a \wedge n \Unit \le a \wedge m \Unit$
  yields $(a \wedge m \Unit)^{-1} \le (a \wedge n \Unit)^{-1}$ 
  by the previous Lemma~\ref{lemma:radicalinverseordering}.  
  Moreover, let $b \coloneqq (a\wedge m\Unit)^{-1} + \Unit / n \in \{a\}'' \cap \A_\Hermitian$,
  then $b$ is also a coercive element of $\A^\bd$, hence invertible.
  From $\Unit / n \le b$ and $(a\wedge m\Unit)^{-1} \le b$ it follows 
  that $b^{-1} \le n\Unit$ and $b^{-1} \le a \wedge m\Unit$ by the previous Lemma~\ref{lemma:radicalinverseordering}.
  So $b^{-1} \le a \wedge n\Unit$ and therefore $(a \wedge n\Unit)^{-1} \le b$ by Lemma~\ref{lemma:radicalinverseordering}
  again.
  Altogether, this shows that 
  $(a \wedge m \Unit)^{-1} \le (a \wedge n \Unit)^{-1} \le (a\wedge m\Unit)^{-1} + \Unit / n$
  for all $m,n\in \NN$ with $n \le m$, so $\NN \ni n \mapsto (a \wedge n \Unit)^{-1} \in \A_\Hermitian$
  is a Cauchy sequence with respect to the uniform metric, and thus converges.
  
  From $0 \le (a\wedge n \Unit) \le a$ it follows that $(a\wedge n \Unit)^2 \le a^2$
  by Proposition~\ref{proposition:radicalOrderSquare} and that $0 \le a - (a\wedge n \Unit)$,
  and Lemma~\ref{lemma:weakOrderUnit} shows that $a - (a\wedge n \Unit) \le a^2/(4n)$.
  So one can apply Lemma~\ref{lemma:invertibleWhenComplete} to the sequence
  $(a\wedge n\Unit)_{n\in \NN}$ with $c \coloneqq a^2$ and $d \coloneqq a^2/4$,
  which shows that $a$ is invertible and $a^{-1} = \lim_{n\to \infty} (a \wedge n \Unit)^{-1}$.
\end{proof}
\section{Square Roots} \label{sec:sqrt}
The usual way to construct absolute values is via square roots of the square.
In order to guarantee the uniqueness of the square roots, it makes sense
to discuss square roots only in radical Archimedean ordered $^*$\=/algebras,
in which Lemma~\ref{lemma:sqrtUnique} applies:
\begin{definition} \label{definition:squareroot}
  Let $\A$ be a radical Archimedean ordered $^*$\=/algebra and $a\in\A^+_\Hermitian$. The
  \neu{square root} of $a$ is (if it exists) the unique element $\sqrt{a} \in \{a\}'' \cap \A_\Hermitian^+$
  fulfilling ${\sqrt{a}\,}^2 = a$.
\end{definition}
\begin{proposition} \label{proposition:absfromsqrt}
  Let $\A$ be a radical Archimedean ordered $^*$\=/algebra and $a\in\A_\Hermitian$ and assume that
  $\sqrt{a^2}$ exists, then $\abs{a}$ exists and is given by $\abs{a} = \sqrt{a^2}$.
\end{proposition}
\begin{proof}
  Note that $\{a\}' \subseteq \{a^2\}'$, therefore $\{a^2\}'' \subseteq \{a\}''$. So 
  $\sqrt{a^2} \in \{a^2\}'' \cap \A_\Hermitian^+ \subseteq \{a\}'' \cap \A_\Hermitian^+$
  and ${\sqrt{a^2}\,}^2 = a^2$ show that $\abs{a} = \sqrt{a^2}$ exists. 
\end{proof}
Especially if $\sqrt{a}$ exists for all positive Hermitian elements $a$ of a radical
Archimedean ordered $^*$\=/algebra $\A$, then $\A$ is a $\Phi^*$\=/algebra
and $\A_\Hermitian^+ = \A_\Hermitian^{++}$, thus every unital $^*$\=/homomorphism
$\Psi \colon \A \to \algebra{B}$ to another ordered $^*$\=/algebra $\algebra{B}$ is automatically
positive as $\Psi(a) = \Psi(\sqrt{a})^2 \in \algebra{B}^+_\Hermitian$. On such algebras,
the order is even uniquely determined, a result that generalizes the uniqueness of the norm of
$C^*$\=/algebras:
\begin{proposition} \label{proposition:uniqueorder}
  Let $\A$ be a radical Archimedean ordered $^*$\=/algebra in which $\sqrt{a}$ exists
  for all $a \in \A_\Hermitian^+$, then the order on $\A_\Hermitian$ is uniquely
  determined in the following sense: Denote the order on $\A_\Hermitian$
  by $\le$, as always. If $\preccurlyeq$ is any order on $\A_\Hermitian$
  such that $\A$ with $\preccurlyeq$ is an ordered $^*$\=/algebra, then $\le$ and $\preccurlyeq$
  coincide.
\end{proposition}
\begin{proof}
  Consider the injective unital $^*$\=/homomorphism $\id_{\A}$ as a map from $\A$ with $\le$ to $\A$
  with $\preccurlyeq$. Then $\id_{\A}$ is automatically positive due to the existence
  of square roots, and as $\A$ with $\le$ is a $\Phi^*$\=/algebra by the previous
  Proposition~\ref{proposition:absfromsqrt} and Corollary~\ref{corollary:Phistarchar}, Proposition~\ref{proposition:openmapping}
  applies and shows that $\id_\A$ is even an order embedding, i.e.\ that $\le$ and $\preccurlyeq$
  coincide.
\end{proof}
Moreover, unital $^*$\=/homomorphisms between such algebras are compatible with square roots:
\begin{proposition} \label{proposition:morphsqrt}
  Let $\A$ and $\algebra{B}$ be two radical Archimedean ordered $^*$\=/algebras and assume that
  the square roots of all positive Hermitian elements in $\A$ and $\algebra{B}$ exist. Moreover,
  let $\Psi \colon \A \to \algebra{B}$ be an (automatically positive) unital $^*$\=/homomorphism.
  Then $\Psi(\sqrt{a}) = \sqrt{\Psi(a)}$ holds for all $a\in \A^+_\Hermitian$.
\end{proposition}
\begin{proof}
  Like Proposition~\ref{proposition:morphveewedgeabs},
  this follows from Lemma~\ref{lemma:sqrtUnique} because
  $\Psi(\sqrt{a}) \in \{\Psi(a)\}' \cap \algebra{B}_\Hermitian^+$
  fulfils $\Psi(\sqrt{a})^2 = \Psi(a)$ and because it is already
  known that $\sqrt{\Psi(a)} \in \{\Psi(a)\}'' \cap \algebra{B}_\Hermitian^+$
  exists.
\end{proof}
In the uniformly complete case, square roots can oftentimes be explicitly constructed:
\begin{lemma} \label{lemma:limitOfSqrt}
  Let $\A$ be a radical and uniformly complete Archimedean ordered $^*$\=/algebra, $\hat{a} \in \A^+_\Hermitian$
  and $(a_n)_{n\in \NN}$ a sequence in $\{\hat{a}\}'' \cap \A^+_\Hermitian$ with limit $\hat{a}$.
  If a sequence $(b_n)_{n\in \NN}$ in $\{\hat{a}\}'' \cap \A^+_\Hermitian$ fulfils
  $b_n^2 = a_n$ for all $n\in \NN$ and is bounded from above by some $c \in \A^+_\Hermitian$,
  then $\sqrt{\hat{a}}$ exists and $\sqrt{\hat{a}} = \lim_{n\to \infty} b_n$.
\end{lemma}
\begin{proof}
  Given $\epsilon \in {]0,1]}$, then there exists an $N\in \NN$ such that
  $-\epsilon^2 \Unit \le a_n - a_N \le \epsilon^2 \Unit$ holds for all $n\in \NN$ with
  $n\ge N$, hence $b_N^2\le b_n^2+\epsilon^2\Unit$ and $b_n^2\le b_N^2+\epsilon^2\Unit$.
  This implies $b_N^2 \le (b_n+\epsilon\Unit)^2$
  and $b_n^2 \le (b_N+\epsilon\Unit)^2$, so $b_N \le b_n+\epsilon\Unit$
  and $b_n \le b_N+\epsilon\Unit$ by Proposition~\ref{proposition:radicalOrderSquare},
  or equivalently $-\epsilon \Unit \le b_n - b_N \le \epsilon \Unit$.
  The sequence $(b_n)_{n\in \NN}$ thus is a Cauchy sequence and 
  has a limit $\hat{b} \coloneqq \lim_{n\to\infty} b_n \in \{\hat{a}\}'' \cap \A_\Hermitian^+$
  as $\{\hat{a}\}'' \cap \A_\Hermitian^+$ is a closed subset of the complete metric space $\A$
  by Propositions~\ref{proposition:closedstuff}.
  
  It only remains to show that $\hat{b}^2 = \hat{a}$, then $\sqrt{\hat{a}} = \hat{b}$ exists.
  For all $\epsilon\in {]0,1]}$ there exists an $N\in \NN$ such that
  $-\epsilon \Unit \le b_n - \hat{b} \le \epsilon \Unit$ and
  $-\epsilon \Unit \le a_n - \hat{a} \le \epsilon \Unit$ hold for all $n\in \NN$ with $n\ge N$.
  The first estimate gives $\hat{b} \le b_n+\epsilon \Unit$ and $b_n \le \hat{b} + \epsilon \Unit$,
  and using Proposition~\ref{proposition:radicalOrderSquare} one obtains
  $\hat{b}^2 \le (b_n + \epsilon \Unit)^2 \le b_n^2 + \epsilon (2b_n + \Unit) \le a_n + \epsilon (2 c + \Unit)$
  and
  $a_n = b_n^2 \le (\hat{b} + \epsilon \Unit)^2 \le \hat{b}^2 + \epsilon (2\hat{b} + \Unit)$.
  Together with the second estimate this yields $\hat{b}^2 - \hat{a} = \hat{b}^2 - a_n + a_n - \hat{a} \le \epsilon (2 c + 2\Unit)$
  and $\hat{a} - \hat{b}^2 = \hat{a} - a_n + a_n - \hat{b}^2 \le \epsilon (2 \hat{b} + 2\Unit)$,
  so $\hat{a} = \hat{b}^2$ because $\A$ is Archimedean.
\end{proof}
\begin{proposition} \label{proposition:sqrtconst}
  Let $\A$ be a radical and uniformly complete Archimedean ordered $^*$\=/algebra and let $a \in \A^+_\Hermitian$.
  If additionally $a$ is uniformly bounded, or $a + \Unit/n$ invertible for all $n \in \NN$,
  then $\sqrt{a}$ exists.
\end{proposition}
\begin{proof}
  As $\A^\bd$ is a $C^*$\=/algebra due to the completeness of $\A$, it is clear that the
  square root of $a$ exists in the uniformly bounded case. This can also be obtained directy by
  combining Lemmas~\ref{lemma:approximateSqrt} and the previous Lemma~\ref{lemma:limitOfSqrt}.
  
  If $a + \Unit/n$ is invertible for all $n\in \NN$, then
  $(a+\Unit/n)^{-1} \in \{a\}''\cap\A_\Hermitian^+ \cap \A^\bd$ for every $n\in \NN$ 
  by Lemma~\ref{lemma:inverseordering} and $\sqrt{(a+\Unit/n)^{-1}} \in \{(a+\Unit/n)^{-1}\}'' \cap \A_\Hermitian^+$
  exists by the first part, and one can thus construct $b_n \coloneqq (a+\Unit/n) \sqrt{(a+\Unit/n)^{-1}}$.
  By Corollary~\ref{corollary:radicalproducts}, $b_n \ge 0$ and one can easily check that even $b_n \in \{a\}'' \cap \A_\Hermitian^+$
  and $b_n^2 = a+\Unit/n$. So the previous Lemma~\ref{lemma:limitOfSqrt} applies and shows that
  $\sqrt{a}$ exists.
\end{proof}
\section{\texorpdfstring{\Sus\=/Algebras}{Su*-Algebras}} \label{sec:su}
Essentially all of the previous results hold for a class of very well-behaved ordered $^*$\=/algebras:
\begin{theorem} \label{theorem:su}
  Let $\A$ be a uniformly complete Archimedean ordered $^*$\=/algebra, then the following
  additional properties are all equivalent:
  \begin{enumerate}
    \item All elements of the form $a+ \I \Unit$ and $a-\I\Unit$ with $a\in \A_\Hermitian$ are invertible. \label{item:maininvbar}
    \item All coercive elements in $\A_\Hermitian$ are invertible, i.e.\ $\A$ is symmetric. \label{item:mainsym}
    \item $\A$ is radical and $\sqrt{a}$ exists for all $a\in \A_\Hermitian^+$. \label{item:mainsqrt}
    \item $\A$ is radical and $\abs{a}$ exists for all $a\in \A_\Hermitian$. \label{item:mainabs}
    \item $\A$ is radical and both $a\vee b$ and $a\wedge b$ exist for all commuting $a,b\in \A_\Hermitian$, i.e.\ $\A$ is a $\Phi^*$\=/algebra. \label{item:mainveewedge}
    \item $\A$ is radical and $a\wedge \Unit$ exists for all coercive $a \in \A_\Hermitian^+$. \label{item:mainsomewedge}
  \end{enumerate}
\end{theorem}
\begin{proof}
  Implication {\refitem{item:maininvbar}} $\implies$ {\refitem{item:mainsym}} follows from
  Corollary~\ref{corollary:symcond} because $\Unit+a^2 = (a+\I\Unit)(a-\I\Unit)$, and
  for the implication {\refitem{item:mainsym}} $\implies$ {\refitem{item:mainsqrt}} one uses
  that every symmetric Archimedean ordered $^*$\=/algebra is automatically radical by
  Proposition~\ref{proposition:symisrad} and thus one can apply Proposition~\ref{proposition:sqrtconst}.
  {\refitem{item:mainsqrt}} $\implies$ {\refitem{item:mainabs}} is Proposition~\ref{proposition:absfromsqrt},
  {\refitem{item:mainabs}} $\implies$ {\refitem{item:mainveewedge}} is Corollary~\ref{corollary:Phistarchar}
  and {\refitem{item:mainveewedge}} $\implies$ {\refitem{item:mainsomewedge}} is trivial.
  Finally, for {\refitem{item:mainsomewedge}} $\implies$ {\refitem{item:maininvbar}}, assume that
  {\refitem{item:mainsomewedge}} holds and let $a\in \A_\Hermitian$ be given. Then $\Unit+a^2$ is coercive
  and $(\Unit+a^2) \wedge (n \Unit) = n ( ((\Unit+a^2)/n) \wedge \Unit)$ exists for all $n\in \NN$ 
  by Proposition~\ref{proposition:veewedge}, so Proposition~\ref{proposition:invconst} applies
  and shows that $\Unit+a^2$ has a multiplicative inverse. But then $(a\pm \I\Unit)^{-1} = (\Unit+a^2)^{-1}(a\mp \I\Unit)$
  exists as well.
\end{proof}
\begin{definition} \label{definition:sus}
  A \neu{\Sus\=/algebra} is a uniformly complete Archimedean ordered $^*$\=/algebra that has one,
  hence all of the equivalent additional properties of the above Theorem~\ref{theorem:su}
\end{definition}
It is obvious that ``\Su'' refers to ``\textit{s}ymmetric and \textit{u}niformly complete''.
Besides the equivalent characterizations given by Theorem~\ref{theorem:su}, a \Sus\=/algebra
$\A$ also has some other interesting properties: 

Because of the existence of square roots, the order on $\A$ is simply the algebraic one, 
i.e. $\A^+_\Hermitian = \A^{++}_\Hermitian$, and every unital $^*$\=/homomorphism 
$\Psi \colon \A \to \algebra{B}$ into an ordered $^*$\=/algebra $\algebra{B}$ is automatically
positive, hence continuous for Archimedean $\algebra{B}$. If $\Psi$ is in addition injective, then it is already an order embedding
by Proposition~\ref{proposition:openmapping}. Proposition~\ref{proposition:uniqueorder} shows
that the order on $\A_\Hermitian$ is the unique one with which $\A$ becomes an ordered $^*$\=/algebra.
Unital $^*$\=/homomorphisms between \Sus\=/algebras are not only compatible with the algebraic
operations and positive, they are also compatible with $\vee$ and $\wedge$, absolute values
and square roots by Propositions~\ref{proposition:morphveewedgeabs} and \ref{proposition:morphsqrt}.
From this point of view, \Sus\=/algebras can be seen just as well-behaved $^*$\=/algebra, not 
necessarily as $^*$\=/algebra with an additional structure, because the order is not
subject to any choice and because unital $^*$\=/homomorphisms between them fulfil all compatibilities
one would reasonably expect.

All these properties are typical for $C^*$\=/algebras and complete $\Phi$\=/algebras, which
are important examples of \Sus\=/algebras: The uniformly bounded \Sus\=/algebras are
precisely the $C^*$\=/algebras, because $\seminorm{\infty}{\argument}$ in this case
is a $C^*$-norm by Proposition~\ref{proposition:Cstar},
and because conversely, $C^*$\=/algebras carry a natural ordering with respect to which they
are well-known to be uniformly bounded and uniformly complete Archimedean ordered $^*$\=/algebras,
and also symmetric due to their well-behaved spectral theory.
The commutative \Sus\=/algebras are the complexifications of complete $\Phi$\=/algebras by 
Proposition~\ref{proposition:comphistar}.

As a consequence, the representation theorems for $C^*$- and $\Phi$\=/algebras partly
apply: All uniformly bounded and closed unital $^*$\=/subalgebras
of a \Sus\=/algebra $\A$ (especially $\A^\bd$) are isomorphic to a $C^*$\=/algebra of bounded operators
on a Hilbert space. Similarly, all commutative and closed unital $^*$\=/subalgebras of $\A$,
that also contain the inverses of all coercive elements (especially bicommutants $S''$ of commutative
subsets $S\subseteq \A_\Hermitian$), are isomorphic to the complexification
of a complete $\Phi$\=/algebra of continuous functions on a compact Hausdorff space with values in the extended real line
by \cite{henriksen.johnson:structureOfArchimedeanLatticeOrderedAlgebras}.
Similar representation theorems specifically adapted to ordered $^*$\=/algebras are developped
in \cite{schoetz:preprintGelfandNaimarkTheorems}. With respect to a well-behaved functional calculus
we note that already the well-behaved polynomial calculus described in Proposition~\ref{proposition:polynomialcalculus}
is far from being trivial, but still remains to be extended to a continuous calculus for \Sus\=/algebras.
Important existing results in this direction are of course the continuous calculus for $C^*$\=/algebras, which also
applies to uniformly bounded elements of \Sus\=/algebras, and the continuous calculus for
$\Phi$\=/algebras from \cite{buskes.pagter.vanRooij:FunctionalCalculusOnRieszSpaces}.
\section{\texorpdfstring{\Sus-Algebras}{Su*-Algebras} of Unbounded Operators} \label{sec:example}
While \Sus\=/algebras are unbounded generalizations of $C^*$\=/algebras and 
non-commutative generalizations of $\Phi$\=/algebras, it still remains to
give examples of \Sus\=/algebras that are not of one of these already well-understood
types. Clearly, such examples should especially be provided by $^*$\=/algebras
of unbounded operators, i.e. by $O^*$\=/algebras like in Example~\ref{example:Ostar}.
In the general case, when $\A\subseteq\Adbar(\Dom)$ is an arbitrary $O^*$\=/algebra on a pre-Hilbert space $\Dom$,
it is clear that $\A$ is an Archimedean ordered $^*$\=/algebra. It remains to find sufficient conditions for $\A$
to be symmetric and uniformly complete. Recall that an element $a$ of a $^*$\=/algebra is called \neu{normal} if
$a^*a=a\,a^*$.
\begin{definition} \label{definition:dominantset}
  Let $\A$ be a quasi-ordered $^*$\=/algebra. Then a \neu{dominant subset} of $\A$
  is a subset $Q\subseteq \A$ of normal and pairwise commuting elements, 
  stable under $\argument^*$, with $\Unit \in Q$ and such that $q^* q$ is coercive and $\lambda q \in Q$
  as well as $qr \in Q$ hold for all $\lambda\in{[1,\infty[}$ and all $q,r\in Q$.
  For such a dominant subset define the subset $Q^\downarrow$ of the commutant of $Q$ in $\A$ as
  \begin{equation}
    Q^\downarrow 
    \coloneqq 
    \set[\big]{a\in Q'}{\exists_{q\in Q} : a^* a \lesssim q^*q \text{ and } a\,a^* \lesssim q^*q }\,.
  \end{equation}
\end{definition}
\begin{lemma} \label{lemma:dominant}
  Let $\A$ be a quasi-ordered $^*$\=/algebra and $q,r \in \A$ commuting and with 
  the property that $q^*q$ and $r^*r$ are coercive.
  Then $\lambda^2 q^* r^* r\,q$ is coercive for all $\lambda \in {]0,\infty[}$ and there
  exists a $\lambda\in {[1,\infty[}$ such that $q^*q+r^*r \lesssim \lambda^2 q^* r^* r\,q$ holds.
\end{lemma}
\begin{proof}
  Let $\epsilon \in {]0,2]}$ be given such that $q^*q \gtrsim \epsilon \Unit$ and
  $r^*r \gtrsim \epsilon \Unit$, then
  $\lambda^2 q^* r^*r\,q$ is coercive for all $\lambda \in {]0,\infty[}$ because
  $\lambda^2 q^* r^*r\,q \gtrsim \lambda^2 \epsilon\, q^*q \gtrsim \lambda^2 \epsilon^2 \Unit$.
  Moreover,
  \begin{equation*}
    (2/\epsilon) \,q^* r^*r\,q 
    = 
    q^*\,(r^* r / \epsilon-\Unit)\,q + r^*\,(q^*q / \epsilon - \Unit)\,r + q^*q + r^*r 
    \gtrsim 
    q^*q + r^*r
  \end{equation*}
  holds. So $q^*q + r^*r \lesssim \lambda^2 q^* r^*r\,q$ if one chooses 
  $\lambda \coloneqq \sqrt{2/\epsilon} \in {[1,\infty[}$.
\end{proof}
\begin{proposition} \label{proposition:domalgebra}
  Let $\A$ be a quasi-ordered $^*$\=/algebra and $Q$ a
  dominant subset of $\A$, then $Q^\downarrow$ is a unital $^*$\=/subalgebra of $Q'$, hence of $\A$,
  and $Q\subseteq Q^\downarrow$.
\end{proposition}
\begin{proof}
  Clearly $\Unit \in Q \subseteq Q^\downarrow$, $\lambda a \in Q^\downarrow$ for all $\lambda \in \CC$
  if $a\in Q^\downarrow$ and $Q^\downarrow$ is stable under $\argument^*$. 
  Let $a,b\in Q^\downarrow$ be given, then there are $q,r \in Q$ that fulfil $a^*a \lesssim q^*q$, $a \,a^* \lesssim q^*q$ and 
  $b^*b \lesssim r^*r$, $b \,b^* \lesssim r^*r$. Then 
  $b^*a^*a\,b \lesssim b^* q^* q\,b = q^* b^* b\,q \lesssim q^* r^*r\,q$
  and similarly also $a\,b\,b^*a^* \lesssim r^*q^*q\,r = q^*r^*r\,q$
  show that $ab\in Q^\downarrow$. Moreover, by the previous
  Lemma~\ref{lemma:dominant}, there exists a $\lambda \in {[1,\infty[}$ for which
  \begin{equation*}
    (a+b)^*(a+b) \lesssim (a+b)^*(a+b) + (a-b)^*(a-b) = 2 (a^*a + b^*b) \lesssim 2(q^*q + r^*r) \lesssim 2\lambda^2 r^* q^* q\,r
  \end{equation*}
  and similarly also
  $(a+b)\,(a+b)^* \lesssim 2\lambda^2 r^* q^* q\,r$ hold, so $a+b \in Q^\downarrow$.
\end{proof}
For example, if $\Dom=\Hilb$ is a Hilbert space, then $Q\coloneqq \set{\lambda\Unit}{\lambda\in{[1,\infty[}}$
is a dominant subset of $\Adbar(\Hilb)$ and $Q^\downarrow = \Adbar(\Hilb)$ is the $^*$\=/algebra of all
bounded, i.e. of all $\seminorm{}{\argument}$\=/continuous linear operators on $\Hilb$.
Similar examples generated by a coercive Hermitian, but not necessarily bounded operator in $\Adbar(\Dom)$ 
on a general pre-Hilbert space $\Dom$ can be constructed analogously. The characterization as a
$^*$\=/algebra of continuous adjointable operators carries over as well:
\begin{definition}
  Let $\Dom$ be a pre-Hilbert space and $a\in\Adbar(\Dom)^+_\Hermitian$. Then define the positive Hermitian
  sesquilinear form $\skal{\argument}{\argument}_a\colon \Dom\times\Dom\to\CC$ as
  \begin{equation}
    \skal[\big]{\xi}{\eta}_a \coloneqq \skal[\big]{\xi}{a(\eta)}
  \end{equation}
  for all $\xi,\eta \in \Dom$. The induced seminorm is denoted by
  $\seminorm{a}{\xi} \coloneqq \sqrt{\skal{\xi}{\xi}_a}$ for all $\xi\in\Dom$.
\end{definition}
Recall that the \neu{graph topology} (see e.g. \cite[Def.~2.1.1]{schmuedgen:UnboundedOperatorAlgebraAndRepresentationTheory})
induced by an $O^*$\=/algebra $\A\subseteq\Adbar(\Dom)$ on the pre-Hilbert space 
$\Dom$ is the locally convex topology defined by all the seminorms $\seminorm{\Unit+a^*a}{\argument}$ 
with $a\in\A$, or equivalently by all the seminorms
$\seminorm{b}{\argument}$ with $b\in\A^+_\Hermitian$ because $b\le \Unit+(\Unit+b)^2$ for 
all $b\in\A^+_\Hermitian$. One can check that this set of seminorms $\seminorm{b}{\argument}$ with $b\in\A^+_\Hermitian$
is even cofinal in the set of all seminorms on $\Dom$ that are continuous with respect to the graph
topology, i.e. for every such continuous seminorm $\textrm{p}$ there exists a
$b\in \A^+_\Hermitian$ such that $\textrm{p} \le \seminorm{b}{\argument}$.
\begin{proposition} \label{proposition:downarrowchar}
  Let $\Dom$ be a pre-Hilbert space and $Q\subseteq\Adbar(\Dom)$ a dominant subset. Then the graph
  topology induced by the $O^*$\=/algebra $Q^\downarrow\subseteq\Adbar(\Dom)$ on $\Dom$ is the locally
  convex topology defined by all the seminorms $\seminorm{q^*q}{\argument}$ with $q\in Q$,
  and the set of seminorms $\seminorm{q^*q}{\argument}$ with $q\in Q$
  is cofinal in the set of all seminorms on $\Dom$ that are continuous with respect to the graph
  topology of $Q^\downarrow$. Moreover, for all $a\in Q'$ the following is equivalent:
  \begin{enumerate}
    \item $a\in Q^\downarrow$, \label{item:downarrowchar:1}
    \item $a$ and $a^*$ are both continuous as maps from $\Dom$ with the graph topology of $Q^\downarrow$ to
    $\Dom$ with the $\seminorm{}{\argument}$\=/topology, \label{item:downarrowchar:2}
    \item $a$ and $a^*$ are both continuous as maps from $\Dom$ with the graph topology of $Q^\downarrow$ to
    itself. \label{item:downarrowchar:3}
  \end{enumerate}
\end{proposition}
\begin{proof}
  The locally convex topology on $\Dom$ defined by all the seminorms $\seminorm{q^*q}{\argument}$ with $q\in Q$
  is clearly weaker than the graph topology of $Q^\downarrow$. Conversely, given 
  $b\in (Q^\downarrow)^+_\Hermitian$, then also $\Unit+b\in (Q^\downarrow)^+_\Hermitian$ and so there
  exists a $q\in Q$ such that $(\Unit+b)^2 \le q^*q$. Consequently 
  $\seminorm{b}{\argument} \le \seminorm{(\Unit+b)^2}{\argument} \le \seminorm{q^*q}{\argument}$, and it 
  follows that the locally convex topology on $\Dom$ defined by all the seminorms $\seminorm{q^*q}{\argument}$
  and the graph topology of $Q^\downarrow$ coincide and that the set of 
  seminorms $\seminorm{q^*q}{\argument}$ with $q\in Q$ is cofinal in the set of all seminorms on 
  $\Dom$ that are continuous with respect to the graph topology of $Q^\downarrow$.
  
  Now let $a\in Q^\downarrow$ be given. Then also $a^* \in Q^\downarrow$  and 
  $\seminorm{}{a(\xi)} = \seminorm{a^*a}{\xi} \le \seminorm{\Unit + a^*a}{\xi}$ and
  $\seminorm{}{a^*(\xi)} = \seminorm{aa^*}{\xi} \le \seminorm{\Unit+aa^*}{\xi}$
  hold for all $\xi \in \Dom$, which shows that $a$ and $a^*$ are both continuous as maps from $\Dom$ with the 
  graph topology of $Q^\downarrow$ to $\Dom$ with the $\seminorm{}{\argument}$\=/topology, i.e.
  {\refitem{item:downarrowchar:1}} implies {\refitem{item:downarrowchar:2}}.
  
  Next assume that some $a\in Q'$ is continuous as a map from $\Dom$ with the graph topology of $Q^\downarrow$ to
  $\Dom$ with the $\seminorm{}{\argument}$\=/topology. Then there is a $q\in Q$ such that 
  $\seminorm{}{a(\xi)} \le \seminorm{q^*q}{\xi}$ holds for all $\xi\in\Dom$, and thus
  $\seminorm{r^*r}{a(\xi)} = \seminorm[\big]{}{r\big(a(\xi)\big)}= \seminorm[\big]{}{a\big(r(\xi)\big)} \le
  \seminorm{q^*q}{r(\xi)} = \seminorm{r^*q^*q\,r}{\xi} $ holds for all $\xi\in\Dom$ and all $r\in Q$
  because $a$ and $r$ commute.
  This shows that $a$ is continuous as a map from $\Dom$ with the graph topology of $Q^\downarrow$ to
  itself. So {\refitem{item:downarrowchar:2}} implies {\refitem{item:downarrowchar:3}}.
  
  Finally, assume that $a\in Q'$ is such that $a$ and $a^*$ are both continuous as maps from $\Dom$ with
  the graph topology of $Q^\downarrow$ to itself. Then there especially exist $q,r\in Q$ such that
  $\seminorm{}{a(\xi)} \le \seminorm{q^*q}{\xi} = \seminorm{}{q(\xi)}$ and
  $\seminorm{}{a^*(\xi)} \le \seminorm{r^*r}{\xi} = \seminorm{}{r(\xi)}$ hold for all $\xi\in\Dom$,
  hence $a^*a \le q^*q \le q^*q + r^*r$ and $a\,a^*\le r^*r \le q^*q+r^*r$. By Lemma 
  \ref{lemma:dominant}, there exists a $\lambda\in{[1,\infty[}$
  such that $a^*a \le t^*t$ and $a\,a^*\le t^*t$ hold for $t\coloneqq \lambda qr \in Q$.
  We conclude that {\refitem{item:downarrowchar:3}} implies {\refitem{item:downarrowchar:1}}.
\end{proof}
Such dominated $^*$\=/algebras of operators are especially interesting
if they are closed:
Recall that an $O^*$-algebra $\A\subseteq\Adbar(\Dom)$ on a pre-Hilbert space $\Dom$
is \neu{closed} if $\Dom$ is complete with respect to the graph topology of $\A$.
\begin{lemma} \label{lemma:inftyseminormOstar}
  Let $\Dom$ be a pre-Hilbert space and $\A\subseteq\Adbar(\Dom)$ an $O^*$\=/algebra
  on $\Dom$, then
  \begin{equation}
    \seminorm{\infty}{a} = \sup_{\xi\in\Dom, \seminorm{}{\xi}=1} \seminorm{}{a(\xi)} \in [0,\infty] \label{eq:inftynormOstar}
  \end{equation}
  holds for all $a\in\A$.
\end{lemma}
\begin{proof}
  The supremum on the right hand side of \eqref{eq:inftynormOstar} is by definition
  the minimum of the set of $\lambda \in {[0,\infty]}$ for which 
  $\seminorm{}{a(\xi)} \le \lambda$ holds for all $\xi\in\Dom$ with $\seminorm{}{\xi}=1$,
  or equivalently, for which $a^*a \le \lambda^2 \Unit$ holds (where $a^*a \le \infty^2 \Unit$ is defined to
  be always true). By Definition~\ref{definition:inftyseminorm} and
  Proposition~\ref{proposition:inftyseminormIsmin}, this minimum is just $\seminorm{\infty}{a}$.
\end{proof}
\begin{proposition} \label{proposition:completeconst}
  Let $\Dom$ be a pre-Hilbert space, $Q\subseteq \Adbar(\Dom)$ a dominant subset and such that
  $Q^\downarrow$ is a closed $O^*$-algebra. Then $Q^\downarrow$ is a uniformly complete Archimedean
  ordered $^*$\=/algebra.
\end{proposition}
\begin{proof}
  Let $(a_n)_{n\in\NN}$ be a Cauchy sequence in $Q^\downarrow$. Then for every
  $\epsilon\in{]0,\infty[}$ there exists an $N\in\NN$ such that 
  $\seminorm{\infty}{a_n-a_{N}} \le \epsilon$ holds for all $n\in\NN$
  with $n\ge N$, and due to the previous Lemma~\ref{lemma:inftyseminormOstar},
  this implies that the estimate 
  $\seminorm{q^*q}{a_n(\xi)-a_{N}(\xi)} = \seminorm[\big]{}{(a_n-a_{N})\big(q(\xi)\big)}
  \le \epsilon \seminorm{}{q(\xi)}$
  holds for all $\xi\in\Dom$, all $q\in Q$ and all $n\in\NN$ with $n\ge N$.
  So for every $\xi \in \Dom$, this together with Proposition~\ref{proposition:downarrowchar}
  shows that $\big( a_n(\xi)\big)_{n\in\NN}$ is a Cauchy 
  sequence with respect to the graph topology 
  of $Q^\downarrow$ on $\Dom$, and thus converges against a limit $\hat{a}(\xi) \in \Dom$.
  The resulting map $\hat{a}\colon \Dom \to \Dom$, $\xi\mapsto \hat{a}(\xi)$ is
  the pointwise limit of the sequence $(a_n)_{n\in\NN}$ and is easily seen to be a 
  linear endomorphism of $\Dom$. The convergence is even uniform in the sense that
  $\seminorm[\big]{}{\big(\hat{a}-a_{N}\big)(\xi)} \le 
  \seminorm[\big]{}{\big(\hat{a}-a_{n}\big)(\xi)} + \seminorm[\big]{}{\big(a_n-a_{N}\big)(\xi)} \le 2\epsilon \seminorm{}{\xi}$ 
  holds for all $\xi\in\Dom$ if $n\in \NN$ with $n\ge N$ is chosen sufficiently large.
    
  As the $^*$\=/involution is continuous
  with respect to $\metric_\infty$ on $Q^\downarrow$, also $(a_n^*)_{n\in\NN}$ is a
  Cauchy sequence in $Q^\downarrow$ and yields a pointwise limit $\tilde{a}\colon \Dom \to \Dom$.
  The inner product $\skal{\argument}{\argument} \colon \Dom\times\Dom \to \CC$
  is $\seminorm{}{\argument}$-continuous as a consequence of the Cauchy Schwarz inequality.
  Using this it is easily seen that $\tilde{a}$ is the adjoint endomorphism of $\hat{a}$,
  so $\hat{a} \in\Adbar(\Dom)$. The previous Lemma~\ref{lemma:inftyseminormOstar}
  together with the above uniform convergence estimate imply that $\hat{a}$ is the
  limit of the sequence $(a_n)_{n\in \NN}$ with respect to $\metric_\infty$ on $\Adbar(\Dom)$.
  
  It only remains to show that $\hat{a}\in Q^\downarrow$: Proposition~\ref{proposition:closedstuff} already
  shows that $\hat{a}\in Q'$. Moreover, there exists an $n\in\NN$ with $\seminorm{\infty}{\hat{a}-a_n} \le 1$,
  i.e. $\seminorm[\big]{}{\big(\hat{a}-a_n\big)(\xi)} \le \seminorm{}{\xi}$ for all $\xi\in\Dom$,
  and a $q\in Q$ fulfilling $a^*_na_n \le q^*q$ and $a_na^*_n \le q^*q$,
  i.e. $\seminorm[\big]{}{a_n(\xi)} \le \seminorm{}{q(\xi)}= \seminorm{q^*q}{\xi}$ and
  $\seminorm[\big]{}{a_n^*(\xi)} \le \seminorm{q^*q}{\xi}$ for all $\xi\in\Dom$.
  So $\seminorm[\big]{}{\hat{a}(\xi)} \le \seminorm[\big]{}{\big(\hat{a}-a_n\big)(\xi)}
  + \seminorm[\big]{}{a_n(\xi)} \le \seminorm{}{\xi} + \seminorm{q^*q}{\xi}$
  for all $\xi\in\Dom$, which shows that $\hat{a}$ is continuous as a map from
  $\Dom$ with the graph topology of $Q^\downarrow$ to $\Dom$ with the $\seminorm{}{\argument}$\=/topology.
  The same is also true for $\hat{a}^*$, so $\hat{a} \in Q^\downarrow$ by 
  Proposition~\ref{proposition:downarrowchar}.
\end{proof} 
This shows that uniform completeness of certain $O^*$\=/algebras can be guaranteed essentially
by a completeness condition on the domain. For the existence of inverses of coercive elements 
one can then make use of Proposition~\ref{proposition:invertibleWhenComplete}:
\begin{theorem} \label{theorem:Susconstruct}
  Let $\Dom$ be a pre-Hilbert space and $Q\subseteq \Adbar(\Dom)$ a dominant subset such that
  $Q^\downarrow$ is a closed $O^*$-algebra. Then $Q^\downarrow$ is a \Sus\=/algebra if and only
  if all $q\in Q$ are invertible.
\end{theorem}
\begin{proof}
  First assume that $Q^\downarrow$ is a \Sus\=/algebra and let $q\in Q$ be given. Then $q^*q$
  is coercive, hence has an inverse, and as $q^* q = q\,q^*$ it follows that $q$ is also 
  invertible with $q^{-1} = (q^*q)^{-1}q^*$.
  
  Conversely, if all $q\in Q$ are invertible, then also all $q^*q$ with $q\in Q$. The previous 
  Proposition~\ref{proposition:completeconst} already shows that $Q^\downarrow$ is uniformly complete.
  Given a coercive $a\in\A^+_\Hermitian$ then there exists a $q\in Q$ such that $(\Unit+a)^2 \le q^*q$,
  and then $\Unit \le (\Unit+a)^2 \le q^*q$ implies $q^*q \le q^*(q^*q)\,q = (q^*q)^2$
  and therefore $a \le (\Unit+a)^2 \le q^*q \le (q^*q)^2$. It follows from 
  Proposition~\ref{proposition:invertibleWhenComplete} that $a$ is invertible, so $Q^\downarrow$ is a \Sus\=/algebra.
\end{proof}
\begin{example} \label{example:Sus}
  Let $\Hilb$ be a Hilbert space and $h$ a self-adjoint (not necessarily bounded) operator on $\Hilb$
  which is coercive in the sense that there exists an $\epsilon\in{]0,\infty[}$ such that
  $\skal{\xi}{h(\xi)} \ge \epsilon \skal{\xi}{\xi}$ holds for all vectors $\xi$ in the domain of $h$.
  Let $\Dom$ be the dense linear subspace of $\Hilb$ consisting of all smooth vectors of $h$, i.e. the
  intersection of the domains of the operators $h^n$ for all $n\in\NN$, and write $q\in\Adbar(\Dom)$ for
  the endomorphism described by the restriction of $h$ to $\Dom$, which is coercive in
  the sense of Definition~\ref{definition:radical}. Then $\Dom$ is complete
  with respect to the locally convex topology defined by all the seminorms $\seminorm{q^{2n}}{\argument}$
  with $n\in \NN_0$. Moreover, $q$ is invertible in
  $\Adbar(\Dom)$ because the inverse operator $h^{-1}\in\Adbar(\Hilb)$ of the self-adjoint coercive 
  $h$ restricts to an endomorphism of $\Dom$ as well. One can check that the set 
  $Q\coloneqq \set[\big]{\lambda q^n}{\lambda\in{[1,\infty[},\, n\in\NN_0}$ is a dominant subset of $\Adbar(\Dom)$,
  and as $q$ and hence all elements of $Q$ are invertible, the previous 
  Theorem~\ref{theorem:Susconstruct} applies and shows that $Q^\downarrow$ is a \Sus\=/algebra.
\end{example}
One important application of the above Example~\ref{example:Sus} is the case where $h$ is the Hamiltonian
operator of a quantum mechanical system. Then the \Sus\=/algebra $Q^\downarrow$ is essentially the algebra 
of all symmetries of this system that are bounded by a power of $h$. Note that choosing $h = \Unit_{\Adbar(\Hilb)}$ produces the $C^*$\=/algebra
$Q^\downarrow = \Adbar(\Hilb)$, so the construction of this Example~\ref{example:Sus} is sufficiently
general to cover (up to taking suitable $^*$\=/subalgebras) at least all $C^*$\=/algebras, and clearly many more.

\end{onehalfspace}
\end{document}